\documentclass[12pt]{amsart}
\usepackage{a4wide}

\usepackage{times}
\usepackage{mathtools, amssymb,amsmath}
\usepackage{graphicx,caption,xspace}
\usepackage{epsfig}
\usepackage{dsfont}
\usepackage[usenames,dvipsnames]{xcolor}
\usepackage{tikz}
\usepackage[T1]{fontenc}
\usepackage[utf8]{inputenc}
\usepackage{array,multirow} 
\usepackage{bbold}
\usepackage{bm,stmaryrd} 
\usepackage{csquotes,enumitem}
\usepackage{cancel}

\usepackage{hyperref,pifont}
\usepackage{todonotes}
\usepackage{dsfont}

\usepackage[capitalize]{cleveref}

\theoremstyle{plain}   

\def\MGP{\mathbb M}
\def\DGP{d_{\text{GP}}}

\newtheorem{lemma}{Lemma}[section]
\newtheorem{theorem}[lemma]{Theorem}
\newtheorem{corollary}[lemma]{Corollary}
\newtheorem{proposition}[lemma]{Proposition}
\newtheorem{definition}[lemma]{Definition}

\theoremstyle{remark}
\newtheorem{remark}[lemma]{Remark}

\newcommand{\Circle}{{\sf Ci}}

\usepackage{navigator} 

\def\RR{\mathbb{R}}

\def\BigO{\mathcal{O}}
\def\eps{\varepsilon}

\def\match{\mathfrak{m}}

\def\dbox{\delta_{\Box}}
\def\SpaceGraphon{\widetilde{\mathcal W_0}}

\DeclareMathOperator{\density}{Dens}
\DeclareMathOperator{\Sample}{Sample}
\DeclareMathOperator{\Leb}{Leb}

\DeclareMathOperator{\disc}{disc}
\DeclareMathOperator{\chords}{chords}
\DeclareMathOperator{\cc}{cc}
\newcommand{\dTV}{\ensuremath{d_{\mathrm TV}}}

\newcommand{\dens}[2]{{\density}(#1,#2)}

\newcommand{\set}[1]{\left\{#1\right\}}
\newcommand{\One}{\bm{1}}
\newcommand{\Wperm}{W^{(\mathrm{perm})}}
\newcommand{\Wcircle}{W^{(\mathrm{circle})}}

\title[Large intersection graphs]{Dense and nondense limits \\for uniform random intersection graphs}

\author[F. Bassino]{Frédérique Bassino}
       \address[FB]{Université Sorbonne Paris Nord, LIPN, CNRS UMR 7030, F-93430 Villetaneuse, France}
       \email{bassino@lipn.univ-paris13.fr}

 \author[M. Bouvel]{Mathilde Bouvel}
   \address[MB]{Université de Lorraine, CNRS, Inria, LORIA, F-54000 Nancy, France}
   \email{mathilde.bouvel@loria.fr}
   
 \author[V. Féray]{Valentin Féray}
  \address[VF]{Université de Lorraine, CNRS, IECL, F-54000 Nancy, France}
  \email{valentin.feray@univ-lorraine.fr}

 \author[L. Gerin]{Lucas Gerin}
       \address[LG]{CMAP, \'Ecole Polytechnique, CNRS, Route de Saclay, 91128 Palaiseau Cedex, France}
       \email{gerin@cmap.polytechnique.fr}

 \author[A. Pierrot]{Adeline Pierrot}
 \address[AP]{LISN, Université Paris-Saclay, Bat. 650 Ada Lovelace, 91405 Orsay Cedex, France}
       \email{adeline.pierrot@lri.fr}


\makeatletter
\@namedef{subjclassname@2020}{%
  \textup{2020} Mathematics Subject Classification}
\makeatother


\begin{document}

\begin{abstract}
We obtain the scaling limits of random graphs drawn uniformly in three  families 
of intersection graphs: permutation graphs, circle graphs, and unit interval graphs. 
The two first families typically generate dense graphs, in these cases we prove a.s.~convergence to an explicit deterministic graphon.
Uniform unit interval graphs are nondense and we prove convergence in the sense of Gromov--Prokhorov 
after normalization of the distances: the limiting object is the interval $[0,1]$ endowed 
with a random metric defined through a Brownian excursion.
Asymptotic results for the number of cliques of size $k$
($k$ fixed) in a uniform random graph in each of these three families
are also given.\\
In all three cases, an important ingredient of the proof is that, 
for indecomposable graphs in each class 
(where the notion of indecomposability depends on the class), 
the combinatorial object defining the graph (permutation, matching, or intervals) is essentially unique. 
\end{abstract}

\maketitle

\section{Introduction}

\subsection{Background: random graphs in classes defined by intersections.}

For a collection of sets $\mathcal{C}$ and a $n$-tuple $S=(s_1,s_2,\dots,s_n )$ of elements in $\mathcal{C}$ (called {\em seed}), the \emph{intersection graph} associated with ${s_1,s_2,\dots,s_n}$ is the  graph with vertex set $\set{v_1,v_2,\dots ,v_n}$ in which two vertices $v_i\neq v_j$ are joined by an edge if and only if $s_i\cap s_j\neq \varnothing$. 
 
Families of intersection graphs associated to natural geometric or combinatorial collections $\mathcal{C}$ have been the object of particular interest. Among other,
the following graph classes have been studied in the literature:
\begin{itemize}
\item \emph{interval graphs}: $\mathcal{C}$ is the collection of intervals on the real line;
\item  \emph{unit interval graphs} (also called \emph{proper interval graphs} or \emph{indifference graphs}): $\mathcal{C}$ is the collection of intervals of length one on the real line;
\item \emph{circle graphs}: $\mathcal{C}$ is the collection of chords of a given circle;
\item \emph{circular arc graphs}: $\mathcal{C}$ is the collection of  arcs of a given circle;
\item \emph{string graphs}: $\mathcal{C}$ is the collection of curves in the plane;
\item \emph{permutation graphs}: $\mathcal{C}$ is the collection of 
straight line segments whose endpoints lie on two parallel lines. (This last definition is equivalent to defining permutation graphs as the \emph{inversion graphs} of permutations).
\end{itemize}
\begin{figure}
\begin{center}
\includegraphics[width=12cm]{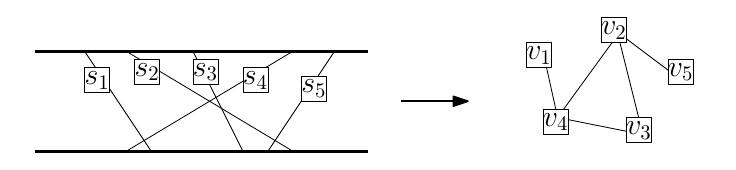}\\
\includegraphics[width=12cm]{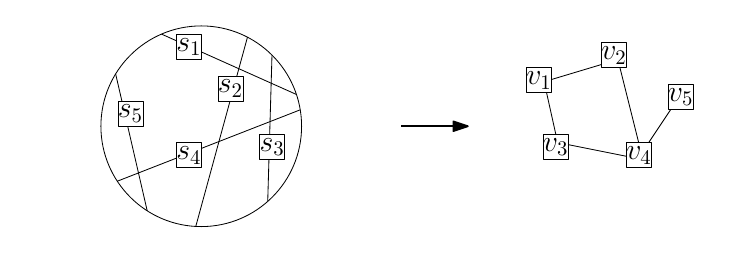}\\
\includegraphics[width=12cm]{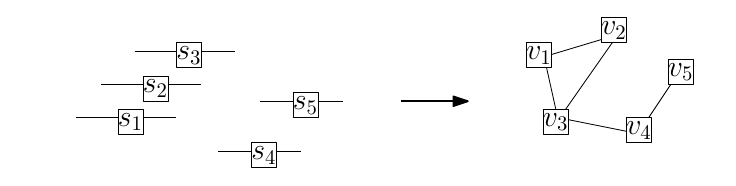}\\
\end{center}
\caption{Illustration of the three families studied in this article.  From top to bottom on the right: a permutation graph, a circle graph, a unit interval graph. In each case one of its representatives $(s_1,s_2,\dots,s_n)$ is shown on the left.
}
\label{fig:Examples}
\end{figure}

We refer the reader to \cref{fig:Examples}
for examples of intersection graphs in three of these families.
The Wikipedia page on the topic \cite{Wikipedia_IntersectionGraphs}
 contains a longer list of graph classes defined by intersection.
Intersection graphs have many applications and 
have been studied in details from an algorithmic point of view,
one problem being to recognize whether a graph is in a given family,
another one to improve the complexity of classical problems
knowing that the input is in the family.
We refer the reader to the books \cite{Golumbic,IntersectionGraphsBook} 
for many such examples.
\smallskip

Intersection graph models have also been of interest in the random graph community.
Here is a selection of references on the topic.
\begin{itemize}
\item Random interval graphs have been introduced and studied by 
Scheinerman \cite{scheinerman88interval} in the 80's -- see also Justicz, Scheinerman and Winkler \cite{scheinerman-et-al90interval}.
The model considered is the \enquote{uniform} model on intervals, \emph{i.e.}~the extremities
$a_i$, $b_i$ of the intervals are taken i.i.d.~uniformly at random in $[0,1]$, conditionally to $a_i < b_i$. We refer also to \cite{DiaconisJansonHolmes} for a discussion
 on graphon limits of such models and further references on random interval graphs.
\item The inversion graph of a uniform random permutation of size $n$ has been recently studied: see Bhattacharya and Mukherjee \cite{DegreePermutationGraphs} for results on the degree sequence 
and Gürerk, I\c slak and Yıldız \cite{StudyRandomPermutationGraphs} for results on the degree distributions, isolated vertices, cliques and connected components.
We also refer to Acan and Pittel \cite{AcanPittelPermutationGraphs} 
for an analysis of inversion graphs of uniform random permutations 
with a fixed number of inversions
(thus fixing the number of edges).
\item In a similar spirit, Acan has studied various properties of the intersection graph of
a uniform random chord diagram in \cite{AcanUniformChordalGraph};
see also Acan and Pittel \cite{AcanPittelChordalGraphs} 
for an analysis of intersection graphs of uniform random chord diagrams 
with a fixed number of crossings (fixing again the number of edges in the graph).
\item The graphon limit of a uniform random string graph has been considered
by Janson and Uzzell \cite{janson17string}, who identified a set of possible limit points,
and conjectured the actual graphon limit.
\item In a slightly different direction, there is an important literature
 around a model called {\em random intersection graphs},
see \cite{RandomIntersectionGraphsReview} and references therein; here a random set is attached to each vertex (most of the time
a uniform random subset with a fixed number of elements of a given set)
and two vertices are connected if their associated sets have a nonempty intersection.
This model is different from the ones cited above in that all graphs can be obtained this way,
and not only graphs from a given family.
\end{itemize}

\subsection{Uniform seeds versus uniform graphs and overview of the results}
\label{Sec:DiscussSeeds}
A noticeable fact in the literature review above is that, in most cases,
the authors consider a natural distribution on the set of \emph{seeds}
(most of the time the uniform one, or the uniform one subject to some size constraint).
This induced a distribution on intersection graphs which is not uniform on the corresponding class. 
(An exception to that is the work of Janson and Uzzell
on string graphs~\cite{janson17string}.)
In contrast, there is a growing literature on uniform random graphs 
in other classes (planar graphs \cite{NoyPlanarICM}
 or graphs embeddable in a given surface~ \cite{RandomGraphsSurface},
 subcritical block-stable classes~\cite{SubcriticalClasses}, 
 perfect graphs~\cite{RandomPerfect}, 
 cographs~\cite{NousCographes,BCographes,NousCographesIndependentSets}, \dots).
For families of intersection graphs however,
studying (or sampling) a uniform graph in the family is often harder than a uniform seed.

It is therefore natural to try to transfer results obtained from the uniform seed model to the uniform graph one, 
and this is the main purpose of our work. 
To this effect, we rely on some known results that, 
in many families of intersection graphs,
there exists some notion of indecomposable graphs,
for which indecomposable graphs can be represented by a {\em unique} seed
(up to some trivial symmetries).
Such uniqueness results have typically been discovered 
in the graph algorithm literature (they are helpful to design recognition algorithms),
and will be useful as well for our purposes.

In this article, we illustrate this approach on three of the families of intersection
graphs listed in the previous section, namely permutation graphs,
circle graphs, and unit interval graphs. 
Interestingly, we need to use a different notion of indecomposability
for each family: {\em prime for the modular decomposition} for permutation graphs,
{\em prime for the split decomposition} for circle graphs, 
and {\em connected} for unit interval graphs.

For each of these three families, we obtain a \enquote{scaling limit} result
for a uniform random graph in the class.
Asymptotic results for the number of cliques of size $k$
($k$ fixed) in a uniform random graph in each of these three families
are also given.

Permutation graphs and circle graphs are typically dense, 
in the sense that the number of edges is quadratic with respect to the number of vertices. We thus use the notion of dense graph limits, a.k.a.~graphon convergence.
 (Definitions and necessary background on graphons will be given in \cref{sec:Graphon}.)
 Namely, we prove that
 a uniform random permutation (resp.~circle) graph
 tends almost surely (a.s.~for short) 
 towards a deterministic limiting graphon $\Wperm$ (resp.~$\Wcircle$).
 The asymptotic result for the number of cliques of size $k$
($k$ fixed) follows as a corollary.
 
 On the other hand, uniform random unit interval graphs with $n$ vertices typically
  have $\Theta(n^{3/2})$ edges.
 We study their limit for the so-called Gromov--Prokhorov (GP) topology,
 which encodes typical distances between randomly sampled vertices.
 (GP convergence is rewiewed in \cref{sec:DefGP}.)
 We prove that, with respect to this topology, a uniform random unit interval graph
 converges towards the unit interval $[0,1]$, endowed with a random metric
 computed from a Brownian excursion.
 The asymptotics of the number of cliques can also been related
 to Brownian excursions
  (though this is not a direct consequence of the GP convergence).
The limiting object and the limiting random variables for renormalized numbers
of cliques are here random,
while they are deterministic for permutation and circle graphs.

\begin{remark}
It would be interesting to study uniform random interval graphs and compare them
with the interval graphs constructed from uniform random intervals considered
by Scheinerman and collaborators \cite{scheinerman88interval,scheinerman-et-al90interval}. Interval graphs are naturally encoded by matchings of the set $\{1,\dots,2n\}$
(the numbers represent the extremities of all intervals in increasing order, and
the endpoints of a given interval are matched together).
 A criterion for unique representability has been given in~\cite[Theorem 1]{hanlon1982interval}, 
but it is intricate and not naturally amenable to the methods of this paper. 

Another interesting family of intersection graphs is that of string graphs:
as mentioned above,
a conjecture regarding the graphon limit of a uniform random string graph
has been formulated by Janson and Uzzell \cite{janson17string}.
But in the case of string graphs, we are not aware of an encoding through 
purely combinatorial objects.

Together with interval and string graphs, the three families of intersection graphs
studied here -- permutation, circle and unit interval graphs -- are the most studied,
explaining our choice to consider them here.
\end{remark}

\subsection{Outline of the article}

\begin{itemize}
\item \cref{sec:results} gives more background on the graph classes studied in this paper. All the results of the paper are  then stated precisely.
\item \cref{sec:Toolboxes} reviews the notions of graphon convergence and Gromov--Prokhorov convergence.
\item Then we separately deal with each family:
\begin{itemize}
\item  \cref{sec:permutation} is devoted to permutation graphs;
\item  \cref{sec:circle} is devoted to circle graphs (for ease of reading, the proofs of two technical results are postponed to the Appendix \ref{Sec:Appendice});
\item  \cref{sec:unit-interval} is devoted to unit interval graphs.
\end{itemize}
\end{itemize}

\section{Results}
\label{sec:results}

\subsection{Permutation graphs}

Permutation graphs have been introduced by Even, Lempel and Pnueli 
in \cite{PermutationGraphs1,PermutationGraphs2}.
For a permutation $\sigma$, we denote by $G^\ell_{\sigma}$ the graph with $|\sigma |$ vertices obtained by the following construction:
\begin{itemize}
\item $G^\ell_{\sigma}$ has vertex set $\set{1,\dots,|\sigma |}$ ;
\item put an edge $i\leftrightarrow j$ if and only if $\{i,j\}$ is an inversion of $\sigma$, \emph{i.e.} $\left(\sigma(i)-\sigma(j) \right)(i-j)<0$.
\end{itemize}
We denote $G_\sigma$ the unlabeled version of $G^\ell_{\sigma}$.
It is called the \emph{inversion graph} of $\sigma$.
A \emph{permutation graph} is a (unlabeled) graph $G$ such that $G = G_\sigma$ for some permutation $\sigma$.
Such a permutation $\sigma$ is then said to \emph{realize} $G$, or is called a {\em realizer} of $G$.

Permutation graphs have been intensively studied from an algorithmic point of view,
see, \emph{e.g.}, \cite{crespelle2010dynamic} and references therein,
or \cite{berard2007reversals} for an application to genomics.

We obtain the following  scaling limit result for a uniform random graph
in this class, with an explicit deterministic limit in the sense of graphons.

\begin{theorem}\label{Th:ConvUnifPermutationGraph}
For each $n \ge 1$, let ${\bm G}_n$ be a  uniform random unlabeled permutation graph with $n$ vertices.
In the space of graphons,
$$
{\bm G}_n \,\stackrel{n\to +\infty}{\longrightarrow}\, \Wperm,\ \text{a.s.},
$$
where $\Wperm$ is defined in~\cref{def:Winv} and \cref{prop:cv_Gsigma}.
\end{theorem}
\begin{figure}[h!]
\begin{center}
\includegraphics[width=7cm]{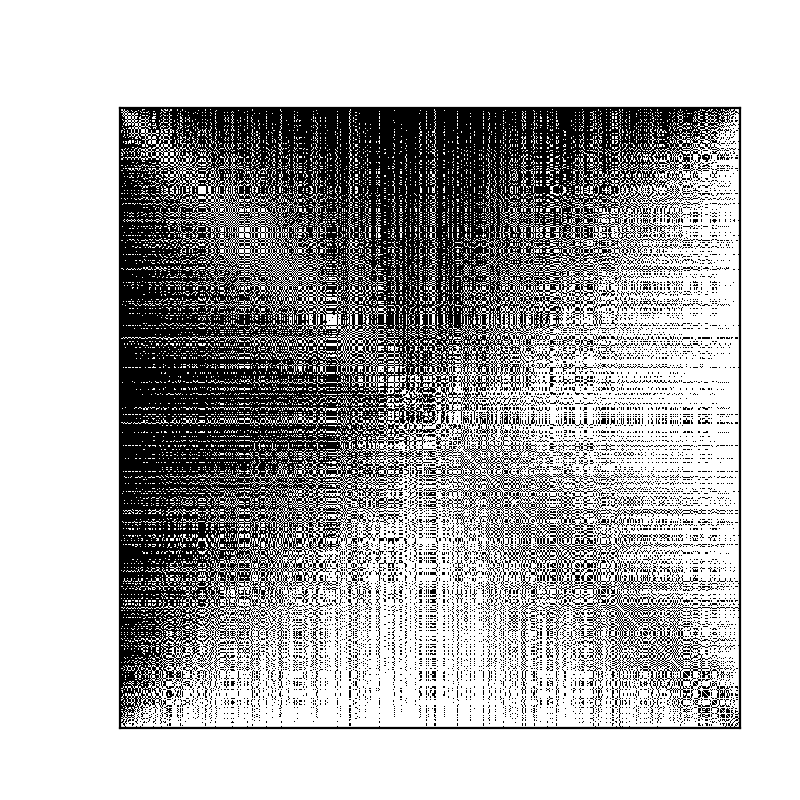}
\ 
\includegraphics[width=7cm]{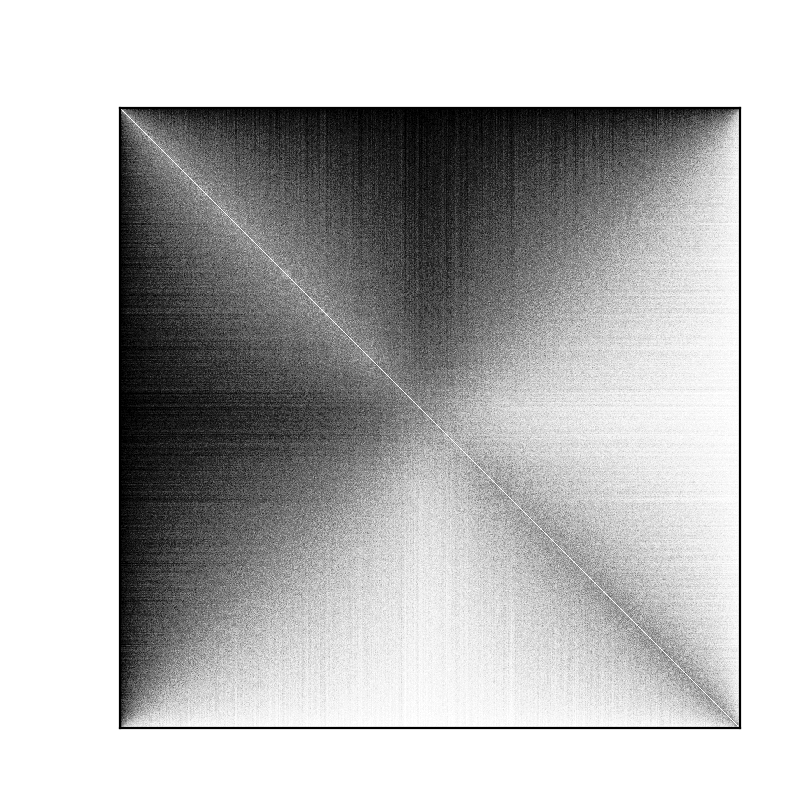}
\end{center}
\caption{Left: The adjacency matrix of the permutation graph $G_{\sigma_n}$ of a uniform random permutation $\sigma_n$ of size $n=1000$, where vertices are  ordered by decreasing degrees. Note that $G_{\sigma_n}$ is not a uniform permutation graph but \cref{prop:cv_Gsigma} ensures that this is a fair approximation of the graphon $\Wperm$. 
Right: The average of $50$ independent adjacency matrices of graphs $G_{\sigma_n}$ for $n=1000$, all ordered by decreasing degrees.}
\label{fig:Wperm}
\end{figure}

To illustrate  \cref{Th:ConvUnifPermutationGraph}, 
we plot in \cref{fig:Wperm} the adjacency matrix of a large random permutation graph
(informally, graphon convergence can be seen as the convergence 
of the rescaled adjacency matrix with a well-chosen order of vertices).

 An interesting feature of graphon convergence is that it encodes the
 convergence of all subgraph counts, correctly renormalized
 (see, \emph{e.g.}, \cite[Chapter 11]{LovaszBook}).
 In the present case, the density of cliques of size $k$ (for all fixed $k$)
 in the limiting graphon  $\Wperm$ 
 can be easily determined (see \cref{prop:MarginalesWinv}).
As a consequence of  \cref{Th:ConvUnifPermutationGraph} we obtain the following estimates:
$$
\frac{\#\set{\text{edges of }{\bm G}_n  }}{\binom{n}{2}} \to \frac12\ \text{a.s.}, \qquad 
\frac{\#\set{\text{triangles of }{\bm G}_n  }}{\binom{n}{3}} \to \frac16\ \text{a.s.},
$$
and more generally for all $k\geq 1$, 
$\displaystyle\frac{\#\set{\text{cliques of size $k$ in }{\bm G}_n  } }{\binom{n}{k}}\to \frac{1}{k!}\ $a.s..

\medskip
\subsection{Circle graphs}

The second family considered in this article is the one of circle graphs, introduced by Even and Itai \cite{CircleGraphs}.
Circle graphs are intersection graphs of chords in a disk. These chords can be seen as a matching between points (corresponding to the endpoints of the chords) along a circle. Circle graphs have been extensively studied from an algorithmic point of view, see the survey \cite{surveyCircle} and references therein. 
The complexity of their recognition posed in \cite{Golumbic} has received considerable attention, see \emph{e.g} \cite{circle-recogn1,circle-recogn2,Gabor}, and has finally been shown to be subquadratic in \cite{circle-recogn3}.
Among other things, circle graphs appear naturally in some routing problems, see \cite{routing-circle}.

As for permutation graphs, we obtain an explicit deterministic limit in the sense of graphons
for a uniform random graph in the class.

\begin{theorem}\label{Th:ConvUnifCircleGraph}
For each $n \ge 1$, let ${\bm G}_n$ be a  uniform random unlabeled circle graph with $n$ vertices.
In the space of graphons,
$$
{\bm G}_n \,\stackrel{n\to +\infty}{\longrightarrow}\, \Wcircle,\ \text{a.s.},
$$
where $\Wcircle$ is defined in~\cref{def:Wcircle} and \cref{prop:gcross}.
\end{theorem}

\begin{figure}
\begin{center}
\includegraphics[width=7cm]{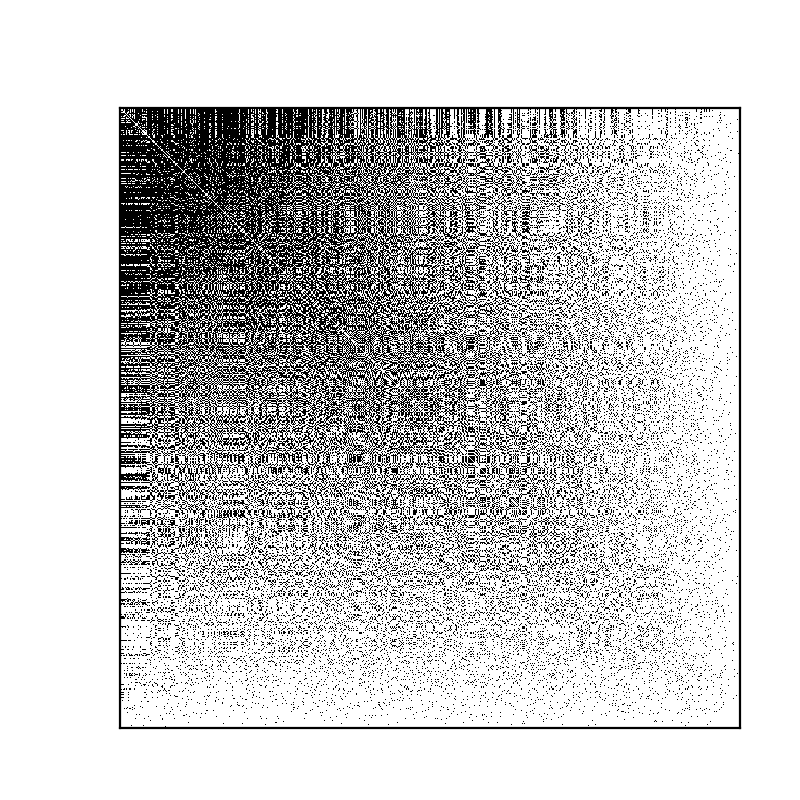}
\ 
\includegraphics[width=7cm]{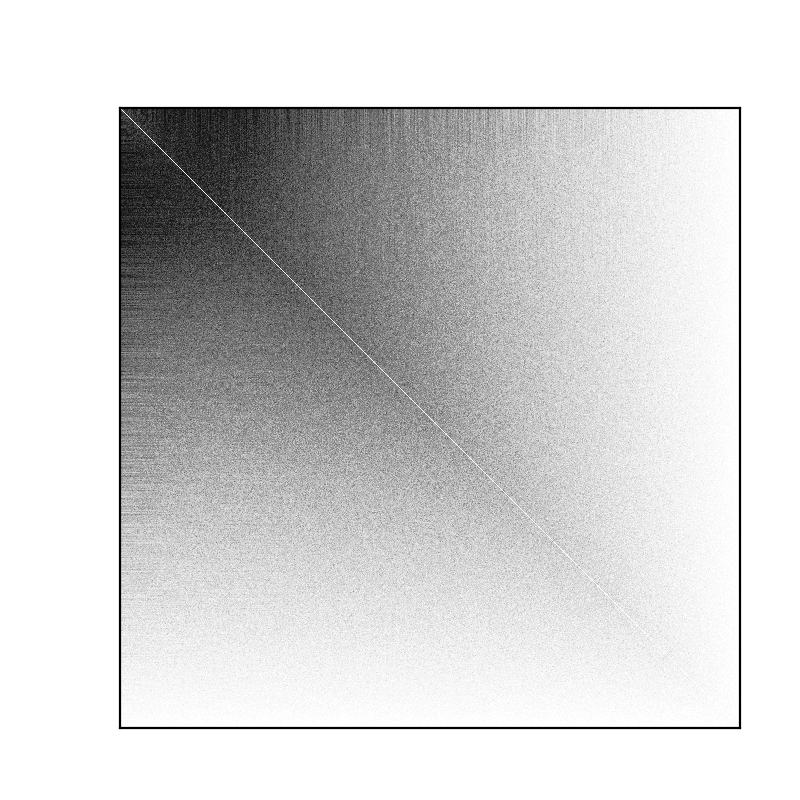}
\end{center}
\caption{Left: The adjacency matrix of the intersection graph $G_{M_n}$ of a uniform matching $M_n$ of size $n=1000$, where vertices are  ordered by decreasing degrees. Note that $G_{M_n}$ is not a uniform circle graph but \cref{prop:gcross} ensures that this is a fair approximation of the graphon $\Wcircle$. 
Right: The average of $50$ independent adjacency matrices of graphs $G_{M_n}$ for $n=1000$, all ordered by decreasing degrees.
}
\label{fig:Wcircle}
\end{figure}

 \cref{Th:ConvUnifCircleGraph} is illustrated in \cref{fig:Wcircle}.
As for $\Wperm$, densities of cliques 
in $\Wcircle$ can be computed easily (see \cref{prop:MarginalesWmatch}),
and  \cref{Th:ConvUnifCircleGraph}
has the following concrete corollary:
$$
\frac{\#\set{\text{edges of }{\bm G}_n }}{\binom{n}{2}} \to \frac13\ \text{a.s.}
,\qquad 
\frac{\#\set{\text{triangles of }{\bm G}_n }}{\binom{n}{3}} \to \frac{1}{15}\ \text{a.s.}
,$$ 
and more generally for all $k\geq 1$, 
$\displaystyle \frac{\#\set{\text{cliques of size $k$ in }{\bm G}_n }}{\binom{n}{k}} \to \frac{2^{k}\, k!}{(2k)!}\ $a.s..

\subsection{Unit interval graphs}
The third family studied in this article is the one of unit interval graphs. 
Unit interval graphs are intersection graphs of intervals of unit length.
From \cite{Rob69} they are equivalent to proper interval graphs, which are intersection graphs
of sets of intervals, where no interval contains another one, and also to claw-free interval graphs
(the claw, also denoted $K_{1,3}$, is the graph with $4$ vertices and $3$ edges such that $1$ vertex is linked with the $3$ other ones).
It is possible to test whether a given graph is a unit interval graph in linear time \cite{LO93}.
We refer the reader to \cite{unit-interval} for other equivalent characterizations
and algorithmic results on unit interval graphs. 

We prove the convergence of unit interval graphs
with a renormalized distance function in the sense of the Gromov--Prokhorov topology.
To this end, for a finite graph $G$ with vertex set $V_G$,
we denote $d_G$ the associated graph distance
and $m_{V_G}$ the uniform distribution on $V_G$.
We recall that a \emph{metric measure space} (called mm-space for short) is a triple $(X,d,\mu)$,
where $(X,d)$ is a metric space and $\mu$ a probability measure on $X$
(for the Borel $\sigma$-algebra induced by $d$).
In the following, $\Leb$ denotes the Lebesgue measure on $[0,1]$.
 \begin{theorem}
 \label{thm:unit-interval}
 Let $\bm G_n$ be a uniform random unlabeled unit interval graph with $n$ vertices.
 The following convergence of random mm-spaces holds in distribution in the Gromov--Prokhorov topology:
 \[(V_{\bm G_n},\tfrac{1}{\sqrt n}\,d_{\bm G_n},m_{V_{\bm G_n}}) \longrightarrow
  ([0,1], \tfrac{1}{\sqrt 2}\,d_{\mathbb e},\Leb),\]
 where $\mathbb e$ is a random Brownian excursion of length 1
 and $d_{\mathbb e}$ is defined by the formula:
 for $x<y$ in $[0,1]$, we have
 \begin{equation}\label{eq:Un_Sur_e}
 d_{\mathbb e}(x,y)=\int_x^y \frac{dt}{\mathbb e(t)}.
 \end{equation}
 \end{theorem} 
 \noindent \cref{thm:unit-interval} is illustrated on \cref{fig:simu_unit}.

\begin{figure}[h!]
\begin{center}
\includegraphics[width=7cm]{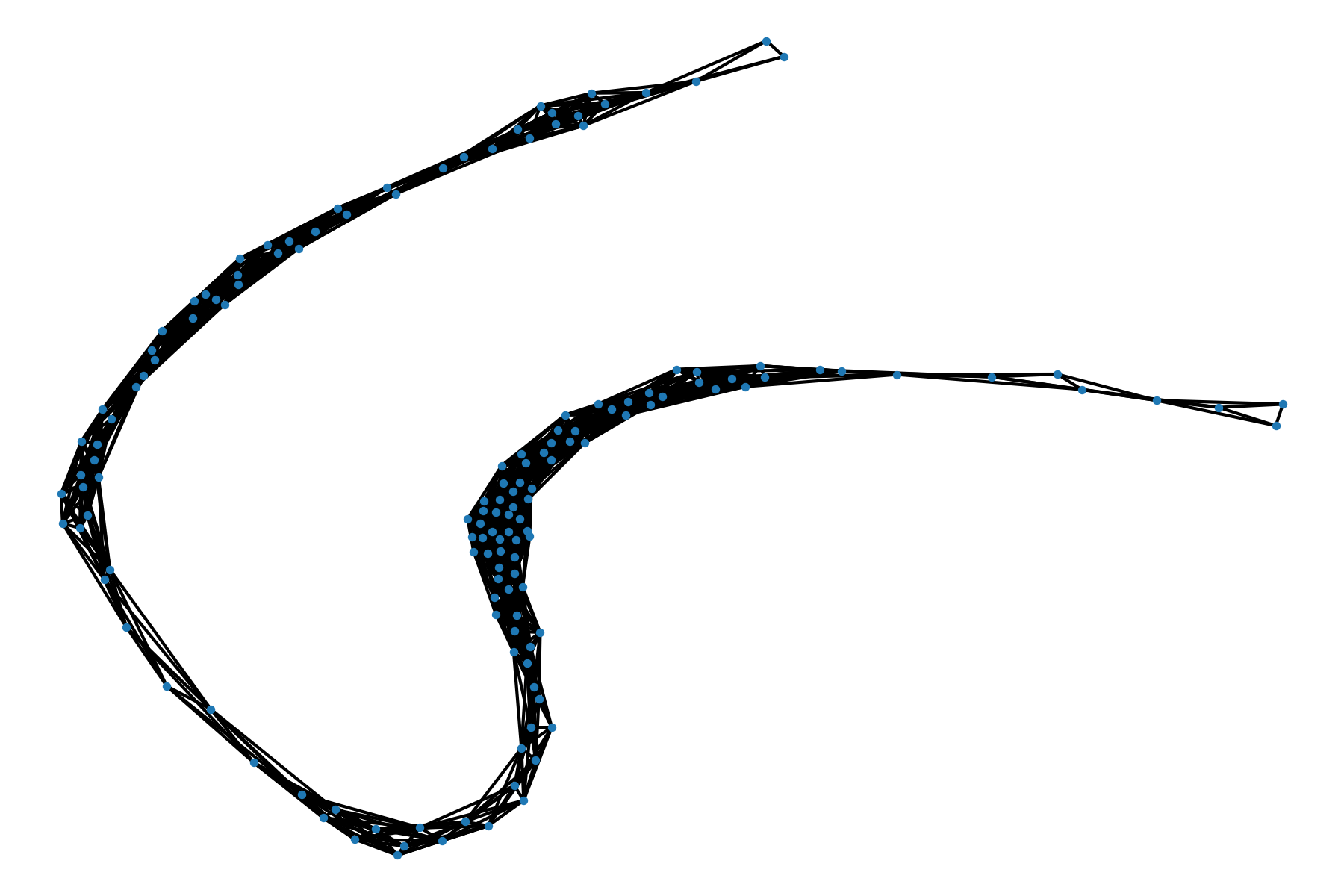}
\ 
\includegraphics[width=7cm]{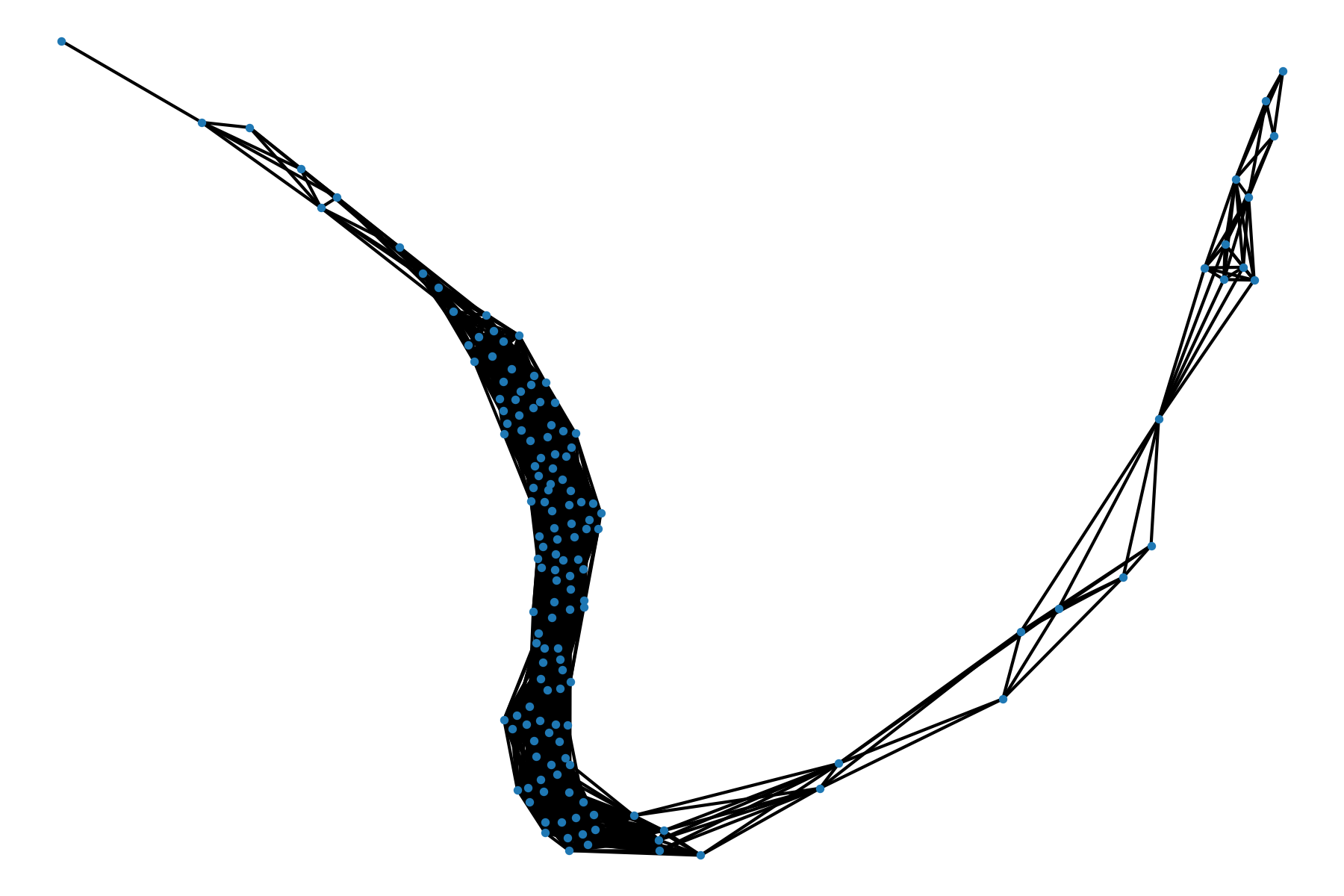}
\end{center}
\caption{Two samples of uniform connected unit interval graphs with $n=150$ vertices. Plots of adjacency matrices are not relevant for nondense graphs so we rather show geometric embeddings of graphs (they were obtained with the \texttt{python} library \texttt{networkx}).
These graph drawings illustrate the fact that the limiting object is  one-dimensional
with a variable density of vertices.
}
\label{fig:simu_unit}
\end{figure}

In particular, graph distances in $\bm G_n$ are typically of order $\sqrt n$.
The intuition behind the nice form of \eqref{eq:Un_Sur_e} is given in \cref{lem:distance_unit}: this combinatorial lemma indeed rewrites the distances in $\bm G_n$ in terms of a sum of inverses of some kind of height function of a Dyck path encoding the graph structure.

Though this does not follow from Gromov-Prokhorov convergence, 
we are also able to describe the asymptotic behavior of the number of cliques of size $k$ ($k$ fixed) in $\bm G_n$.
We obtain the following theorem.

\begin{theorem}
\label{thm:number-of-cliques}
Let $\bm G_n$ be a uniform random unit interval graph with $n$ vertices.
Let also $\mathbb e$ be a Brownian excursion and $X_k=\int_{0}^1 \mathbb e(t)^k dt$.
Then for any $K \ge 1$, we have the following joint convergence in distribution:
\[ \left( \frac{\#\set{\text{cliques of size $k$ in }\bm G_n }}{n^{\frac{k+1}2}} \right)_{2 \le k \le K} 
\,\stackrel{n\to +\infty}{\longrightarrow}\,
\left( \frac{2^{\frac{k-1}2}}{(k-1)!}  X_{k-1}\right)_{2 \le k \le K},\]
where the $X_k$ in the right-hand side are computed from the same 
realization of the excursion $\mathbb e$.
\end{theorem}
Observe that the renormalization factor differs from the case of permutation graphs and circle graphs. In particular for $k=2$, \cref{thm:number-of-cliques} says that $\bm G_n$ has typically $\Theta(n^{3/2})$ edges and thus confirms that it is  nondense.

The random variables $(X_k)_k$ have been studied in the probabilistic literature:
we refer to the survey \cite{JansonSurveyBrownianArea} for an extensive study and many bibliographic pointers
regarding $X_1$, to \cite{nguyen2004area} for formulas for the joint Laplace transform and joint moments
of $(X_1,X_2)$ and finally to \cite{richard2009excursion-moments} for the computation of joint moments 
of~$(X_1,\dots,X_K)$ for general $K$.
Unlike in the case of permutation and circle graphs, these limiting
random variables are not deterministic,
\emph{i.e.}~there is no concentration of the number of cliques of size $k$ in $\bm G_n$
around its mean.

\section{Preliminary: convergence of random graphs}\label{sec:Toolboxes}

In this section, we present the two notions of convergence for random graphs used in this paper,
namely graphon convergence and Gromov--Prokhorov convergence.
The necessary material for the proofs of our main theorems is recalled.

\subsection{Dense graph limits and graphons}\label{sec:Graphon}

\medskip

Recall that a function $\phi:[0,1] \to [0,1]$ is Lebesgue-preserving,
if, for any uniform random variable $U$ on $[0,1]$, the variable $\phi(U)$ is
also uniform on $[0,1]$.

\begin{definition}\label{defi:Graphon}
A graphon is an equivalence class of symmetric functions $[0,1]^2 \to [0,1]$,
under the equivalence relation $\sim$, where $w \sim u$ if 
there exists an invertible Lebesgue-preserving function $\phi:[0,1] \to [0,1]$
such that  $w(\phi(x),\phi(y)) = u(x,y)$ for almost every $x,y\in[0,1]$.
\end{definition}

Intuitively, a graphon is a continuous analogue of the adjacency matrix of a graph,
viewed up to relabeling of its continuous vertex set. Finite graphs are naturally embedded into graphons as follows.
\begin{definition}\label{defi:Graphe->Graphon}
	The graphon $W_G$ associated to a labeled graph $G$ with $n$ vertices (labeled from $1$ to $n$) is the equivalence class of the function $w_G:[0,1]^2\to [0,1]$ where
	\[w_G(x,y) = A_{\lceil nx\rceil,\lceil ny\rceil} \in \{0,1\}\]
and $A$ is the adjacency matrix of the graph $G$.
\end{definition}
Since any relabeling of the vertex set of $G$ gives the same graphon $W_G$, 
the above definition immediately extends to unlabeled graphs. 
The space of graphons is endowed with a pseudo-metric $\dbox$ called the {\em cut metric} (see all definitions in \cite[Ch.8]{LovaszBook}). 
Denote by $\SpaceGraphon$ the space of graphons where we identify $W,W'$ whenever  $\dbox(W,W') =0$.
The metric space $(\SpaceGraphon,\dbox)$ is compact \cite[Theorem 9.23]{LovaszBook}.
In the sequel, we think of graphons as elements in $\SpaceGraphon$ and
convergences of graphons are to be understood with respect to the distance $\dbox$. 

Accordingly, given a sequence of graphs $(H_n)_n$, we say that $(H_n)_n$ converges to a graphon $W$ when $(W_{H_n})_n$ converges to $W$. 
\medskip

\subsubsection*{Sampling from graphons and subgraph densities}\label{Sec:SubgraphDensities}

Consider a graphon $W$ and one of its representatives $w:[0,1]^2\to [0,1]$.
Denote by $\Sample_k(W)$ the unlabeled random graph built as follows: 
$\Sample_k(W)$ has vertex set $\{v_1,v_2,\dots,v_k\}$ and,
letting $\vec{X}^k=(X_1,\dots,X_k)$ be i.i.d.~uniform random variables in $[0,1]$,
we connect vertices $v_i$ and $v_j$ with probability $w(X_i,X_j)$
(these events being independent, conditionally on $(X_1,\cdots,X_k)$).

Then the
density of a graph $g$ with $k$ vertices in a graphon $W$ is defined as
\begin{align*}
\density(g,W) &= \mathbb{P}(\Sample_k(W)=g).
\end{align*}

As briefly mentioned in the introduction, a remarkable aspect of graphon convergence
is that it is equivalent to the convergence of all subgraph densities,
see, \emph{e.g.},~\cite[Chapter 11]{LovaszBook}.

In the present article, we will use the fact that for large $k$ a graphon is well approximated by $\Sample_k(W)$, w.r.t.~the distance $\dbox$. 
More precisely we have the following.

\begin{lemma}\label{lem:EstimeeSample}
For every graphon $W$ and every $\eps>0$, we have
\begin{equation}\label{eq:prob_summable}
  \sum_{k=1}^{+\infty} \mathbb P(\delta_{\Box}(\Sample_k(W),W) \ge \eps) < +\infty.
\end{equation}
Consequently, the sequence of random graphs  $(\Sample_k(W))_k$ converges a.s.~to $W$ in the sense of graphons,
 and any graphon $W$ is uniquely determined by the distribution of its samples.
\end{lemma}
\begin{proof}
From \cite[Lemma 10.16]{LovaszBook}, there exists a constant $c>0$ such that for every graphon $W$ and every $k\geq 1$,
\begin{equation}\label{eq:distance_sample}
\mathbb P\bigg[\delta_{\Box}(\Sample_k(W),W) \ge \frac{c}{\sqrt{\log k}}\bigg]\leq \exp\left(
\frac{-k}{2\log(k)}\right).
\end{equation}
Using this, for any fixed $\eps>0$, we have, for $k$ large enough,
\[ \mathbb P\big[\delta_{\Box}(\Sample_k(W),W) \ge \eps\big]
\le  \mathbb P\bigg[\delta_{\Box}(\Sample_k(W),W) \ge \frac{c}{\sqrt{\log k}}\bigg]\leq \exp\left(
\frac{-k}{2\log(k)}\right),\]
so that $\mathbb P(\delta_{\Box}(\Sample_k(W),W) \ge \eps)$ is summable in $k$, as claimed.

The almost sure convergence of $(\Sample_k(W))_k$ to $W$ can be found, {\em e.g.}, in 
\cite[Proposition 11.32]{LovaszBook}, but let us explain how to deduce it from \eqref{eq:prob_summable},
since we will reuse this (classical) argument later.
Using Borel--Cantelli lemma, \eqref{eq:prob_summable} implies that $\mathbb P(A_\eps) = 1$,
 where $A_\eps$ is the event \enquote{$\exists K, \forall k \geq K, \delta_{\Box}(\Sample_k(W),W) < \eps$}.
 Taking a countable intersection, we have $\mathbb P( \bigcap_{\eps \in \mathbb Q_+^*} A_\eps) = 1$,
 implying that $\delta_{\Box}(\Sample_k(W),W)$ converges a.s.~to $0$.
\end{proof}

\subsection{Metric measure spaces and the Gromov--Prokhorov topology}\label{sec:DefGP}
We now formally introduce the Gromov--Prokhorov topology used in \cref{thm:unit-interval}.
Note that the material below will only be used in \cref{sec:unit-interval}.

\begin{definition}
A metric measure space (called mm-space for short) is a triple
$(X,d,\mu)$, where $(X,d)$ is a complete and separable metric space 
and $\mu$ a Borel probability measure on $X$.
\end{definition}
A finite connected graph $G$ can be seen as a mm-space $(V_G,d_G,m_{V_G})$, where $V_G$ is the vertex set of the graph, $d_G$ is the graph distance, and $m_{V_G}$ the uniform distribution on $V_G$.\smallskip

We let $\MGP$ be the set of all mm-spaces\footnote{To avoid Russell's paradox, throughout the section,
we actually take the set of mm-spaces {\em whose elements are not themselves metric spaces}.}, modulo the following relation:
$(X,d,\mu) \sim (X',d',\mu')$ if there is an isometric embedding  $\Phi:X \to X'$
such that the pushforward\footnote{Recall that $\Phi_*$ is defined as follows: $\forall A \subseteq X'$, $\Phi_*(\mu)(A)=\mu (\Phi^{-1}(A))$ where $\Phi^{-1}(A) = \{x \in X \mid \Phi(x) \in~A\}$.} measure $\Phi_*(\mu)$ satisfies $\Phi_*(\mu)=\mu'$. 
Note that $\Phi$ does not need to be invertible, so that we need to consider
the closure of that relation by symmetry and transitivity.

On the set $\MGP$, one can define a distance as follows.
First we recall the notion of Prokhorov distance.
For Borel probability measures $\mu$ and $\nu$ on the same metric space $Y$, we set
\[d_P(\mu,\nu)=\inf\big\{\eps>0: \, \mu(A)\le \nu(A^\eps)+\eps \text{ and }\nu(A)\le \mu(A^\eps)+\eps
\text{ for all measurable }A\subseteq Y\big\},\]
where $A^\eps$ is the $\eps$-halo of $A$, \emph{i.e.} the set of all points at distance at most $\eps$ of $A$.
It is well-known that this distance metrizes the weak convergence of probability measures.

Next we define the Gromov--Prokhorov (GP) distance which induces a topology on mm-spaces. 
Given two mm-spaces $(X,d,\mu)$ and  $(X',d',\mu')$, we set
\[\DGP\big( (X,d,\mu), (X',d',\mu') \big)
= \inf_{(Y,d_Y),\Phi,\Phi'}  d_P\big( \Phi_*(\mu), \Phi'_*(\mu') \big),\]
where the infimum is taken over isometric embeddings $\Phi:X \to Y$ and $\Phi':X' \to Y$
into a common metric space $(Y,d_Y)$.
One can prove \cite[Section 5]{GromovProkhorov} that $\DGP$ is a distance on $\MGP$ 
and that the resulting metric space $(\MGP,\DGP)$ is complete and separable.

A nice property (which we will however not use in this paper)
is that the convergence of a sequence of mm-spaces $(X_n,d_n,\mu_n)$ for the $\DGP$ distance is
equivalent to the convergence, for any $k$, of the matrix $(d_n(x_i,x_j))_{1 \le i,j \le k}$
recording the distances between $k$ independent random elements of $X_n$,
having distribution $\mu_n$ (see \cite[Th.5]{GromovProkhorov}  or \cite[Sec.4]{JansonSurveyGromovPorhorov}).

Instead of $\DGP$ we will use in our proof another distance, which has been shown by L\"ohr \cite{loehr2013equivalence} to induce the same topology.

\begin{definition}[Box distance]\label{def:BoxDistance}
The box distance $\Box$ between two mm-spaces is defined as
$$
\Box((X,d,\mu),(X',d',\mu'))=\inf_{(R,\nu)} \ \ \max(\disc(R),1-\nu(R))
$$
where 
\begin{itemize}
\item the infimum is taken over all pairs $(R,\nu)$ where
\begin{itemize}
\item $R$ is a Borel subset of $X \times X'$  
\item $\nu$ is a coupling of $\mu$ and $\mu'$, \emph{i.e.} a Borel measure on $X \times X'$ whose marginals are $\mu$ and $\mu'$;
\end{itemize}
\item $\disc(R)$ is the \emph{discrepancy} defined by 
$$
 \disc(R) = \sup_{(x_1,x'_1),(x_2,x'_2) \in R} 
\left|  d(x_1,x_2) -  d'(x'_1,x'_2) \right|.
$$
\end{itemize}
\end{definition}

As said above, we have the following result.
\begin{theorem}[Corollary 3.2 in \cite{loehr2013equivalence}, see also \cite{JansonSurveyGromovPorhorov}]\label{th:Loehr}
Distances $\Box$ and $\DGP$ induce the same topology on $\MGP$.
\end{theorem}
Finding a good upper bound on the distance $\Box$ requires to construct a 
pair $(R,\nu)$ with $\nu(R)$ large and a small discrepency.
This is often easier than constructing isometric embedding $\Phi$ and $\Phi'$
of $X$ and $X'$ into a common metric space $Y$, as required to find an upper bound
for $\DGP$. This explains that \cref{th:Loehr} is often useful to prove 
GP convergence, and our proof of \cref{thm:unit-interval} follows this path.

\section{Permutation graphs}
\label{sec:permutation}

The main goal of this section is to prove Theorem~\ref{Th:ConvUnifPermutationGraph}: uniform random permutation graphs $(\bm G_n)$ converge to the graphon $\Wperm$ defined in \cref{def:Winv} and \cref{prop:cv_Gsigma} below. 

Our proof starts by observing that the inversion graph of a uniform random permutation (which is \emph{not} a uniform random permutation graph) converges to $\Wperm$ (see again \cref{prop:cv_Gsigma}).
In order to transfer this result onto uniform permutation graphs, we will go through \emph{modular-prime} permutation graphs (introduced in \cref{Sec:ModularPrime} below). Indeed,  one can control the number of realizers of a modular-prime permutation graph (\cref{prop:nb-antecedents-simple}). We combine all this in \cref{ssec:mainproofperm} to prove Theorem~\ref{Th:ConvUnifPermutationGraph}.
Finally, \cref{ssec:nb-clique-perm} deduces from Theorem~\ref{Th:ConvUnifPermutationGraph} an asymptotic estimate of
the number of cliques of size $k$ in a uniform permutation graph.

\subsection{Graphon limit of the inversion graph of a uniform random permutation}

For any $n$, let $\sigma_n$ denote a uniform random permutation of size $n$. 
In this section we determine the limit in the sense of graphons of its inversion
graph $G_{\sigma_n}$. 

\begin{definition}
\label{def:Winv}
Let 
$$
\begin{array}{r c c c}
\Psi: & [0,1] & \to & [0,1]^2\\
        & x   & \mapsto & (\Psi_1(x),\Psi_2(x))
\end{array}
$$
be any Lebesgue-preserving measurable function, meaning $\Psi_*(\mathrm{Lebesgue}_{[0,1]})=\mathrm{Lebesgue}_{[0,1]^2}$~.
The graphon $\Wperm_\Psi$ is defined as (the equivalence class of) 
$$
\Wperm_\Psi(x,y)=\mathbb{1}_{(\Psi_1(x)-\Psi_1(y))(\Psi_2(x)-\Psi_2(y))<0} \qquad \text{ for all }x,y\in [0,1].
$$
\end{definition}
In words, $\Wperm_\Psi$ takes values in $\{0,1\}$ and is such that $\Wperm_\Psi(x,y)=1$ exactly when the two points $\Psi(x),\Psi(y)$ form an inversion in the unit square, \emph{i.e.}~when one of the two points is at the bottom right of the other.

 \begin{proposition}
  \label{prop:cv_Gsigma}
  The equivalence class of $\Wperm_\Psi$ is independent of the choice of
the Lebesgue-preserving function $\Psi$.
Moreover, let us consider, for every $n\geq 1$,
   a uniform random permutation   $\sigma_n$ of size $n$. Then
$$
G_{\sigma_n} \stackrel{\text{a.s.}}{\to} \Wperm,
$$
where $\Wperm:=\Wperm_\Psi$ for an arbitrary  Lebesgue-preserving function $\Psi$.
\end{proposition}

\begin{proof}
We first identify for every $n$ the distribution of  $\mathrm{Sample}_n(\Wperm_\Psi)$.
By definition,
it is constructed by taking $X_1$, \dots, $X_n$ independently and uniformly in $[0,1]$, 
and by connecting $v_i$ and $v_j$ if and only if $\Wperm_\Psi(X_i,X_j)=1$ 
(recall that $\Wperm_\Psi$ is $\{0,1\}$-valued).
Let $(a_i,b_i)=\Psi(X_i)$. Since $\Psi$ is Lebesgue-preserving,
the $n$ points $(a_1,b_1),\dots,(a_n,b_n)$ are i.i.d.~uniform points in $[0,1]^2$.
Up to relabeling simultaneously $(X_i)_{i\le n}$, $(a_i)_{i\le n}$ and $(b_i)_{i\le n}$,
 we assume $a_1<\dots<a_n$.
Then there exists a unique permutation $\pi$ such that
\[ b_{\pi^{-1}(1)} < \dots < b_{\pi^{-1}(n)}. \]
The permutation $\pi$ is a uniform permutation of size $n$. 
Then we have
\[ \Wperm_\Psi(X_i,X_j)=1 \Longleftrightarrow (a_i-a_j)(b_i-b_j)<0 \Longleftrightarrow (i-j)(\pi(i)-\pi(j) )<0, \]
\emph{i.e.} $(i,j)$ is an inversion of $\pi$.
Thus $\mathrm{Sample}_n(\Wperm_\Psi)$ is the inversion graph of $\pi$.
Since $\pi$ is uniform, 
\begin{equation}\label{eq:LoiSample_Wperm}
\mathrm{Sample}_n(\Wperm_\Psi)=G_{\pi}
\stackrel{\text{(d)}}{=}G_{\sigma_n}.
\end{equation}
In particular, we see that, for any $n \ge 1$, the distribution of $\mathrm{Sample}_n(\Wperm_\Psi)$ is independent of $\Psi$. 
Since a graphon is determined by the distribution of its samples,
$\Wperm_\Psi$ is indeed independent of $\Psi$, as claimed.

Using \cref{eq:LoiSample_Wperm}, we get that for every fixed $\eps>0$ 
$$
\mathbb{P}(\delta_{\Box}(G_{\sigma_n},\Wperm) \ge \eps)=\mathbb{P}(\delta_{\Box}(\mathrm{Sample}_n(\Wperm),\Wperm) \ge \eps),
$$
and the right-hand side is summable by \cref{lem:EstimeeSample}. 
Using the Borel--Cantelli Lemma as in the proof of \cref{lem:EstimeeSample}, we get that $G_{\sigma_n}\stackrel{\text{a.s.}}{\to} \Wperm$, concluding the proof.
\end{proof}

\begin{remark}
\label{rmk:graphons-on-general-S}
One can more generally define graphons as (equivalence classes)
of measurable functions $\mathcal S \times \mathcal S \to [0,1]$, where $\mathcal S$
is any probability space, see, e.g., \cite[Chapter 13]{LovaszBook}.
With this convention, the graphon $\Wperm$ has a simple representative $W$,
using $\mathcal S=[0,1]^2$, namely
\[W((x,x'),(y,y')) =\mathbb{1}_{(x -y )(x'-y')<0}.\]
A similar remark holds for the limit $\Wcircle$ 
of circle graphs defined later in \cref{def:Wcircle}.
\end{remark}

\subsection{Modular-prime permutation graphs, simple permutations and number of realizers}
\label{Sec:ModularPrime}

The random permutation graph $G_{\sigma_n}$ is not a uniform random graph taken among all permutation graphs with $n$ vertices, since some permutation
graphs have more permutations realizing them than others.
Our next goal is to transfer the convergence result for $G_{\sigma_n}$ (\cref{prop:cv_Gsigma})
to a uniform random permutation graph $\bm G_n$ on $n$ vertices.
To do that, we use the notion of modular-prime graphs,
and show that the number of realizers is well-controlled for these graphs.
\medskip

\begin{definition}
A module $M$ in a graph $G$ is a subset of vertices of $G$ 
such that for all $m$, $m'$ in $M$, and $u$ not in $M$
we have that, either both $\{m,u\}$ and $\{m',u\}$ are edges of $G$, or none of them is.

A graph $G$ is called modular-prime if it contains no nontrivial modules,
\emph{i.e.}~no modules other than $\emptyset$, $V_G$ and the singletons $\{v\}$ (for $v\in V_G$).
\end{definition}
 There exists a corresponding notion for permutations, introduced
 by Albert and Atkinson in \cite{AA05}.
 We use the standard notation $[n]:=\{1,\dots,n\}$.
 \begin{definition}
 An interval $I$ in a permutation $\sigma$ is a set of contiguous indices $I$,
 whose image $\sigma(I)$ by $\sigma$ is also contiguous.
 
 A permutation $\sigma$ of size $n$ is called {\em simple} if it has no nontrivial intervals,
\emph{i.e.}~no intervals other than $\emptyset$, $[n]$ and the singletons $\{i\}$ (for $i \in [n]$).
 \end{definition}
 
For example, $I=\set{3,4,5}$ and $I'=\set{6,7}$ are a nontrivial intervals of the permutation $\sigma$ defined by
$$
\begin{array}{r c c c c c c c}
i: & 1 & 2& 3 &4 &5 &6 &7\\ 
\sigma(i): & 7 & 1& 4&6 &5 &2 &3
\end{array} 
 $$

It is easily seen that if $I$ is an interval in $\sigma$,
then the corresponding vertices form a module in $G_\sigma$. 
The converse is not true in general but it holds that $G_\sigma$ is modular-prime if and only if $\sigma$ is simple. (This observation is due to F.\ de Montgolfier \cite{Montgolfier}, see also \cite[Lemma 20]{HabibPaul}.) 

Moreover, we have the following remarkable property.
\begin{proposition}\label{prop:nb-antecedents-simple}
For any modular-prime permutation graph $G$, 
there are at most $4$ permutations $\tau$, all simple, such that $G=G_\tau$.
\end{proposition}

This result is not explicitly stated in the literature,
but follows easily combining various results on  comparability and permutation graphs, 
all recalled in~\cite{Golumbic}. 
In the remaining part of this section,
we explain how results in~\cite{Golumbic} imply \cref{prop:nb-antecedents-simple}.
We also refer to \cite{crespelle2010dynamic} for a general discussion
on how to construct the set of realizers of a permutation graph
 using its modular decomposition.

We first state a useful characterization of permutation graphs given in~\cite{Golumbic}.
For this, recall that the complement of a graph $G=(V,E)$ is $\bar{G} = (V,\bar{E})$ where $\{u,v\} \in \bar{E}$ if and only if $u \neq v$ and $\{u,v\} \notin E$.  
Recall also that a graph $G=(V,E)$ is  a comparability graph if and only if 
there exists a partial order $\prec$ on $V$ such that $\{u,v\}$ is an edge of $G$ if and only if $u \prec v$ or  $v \prec u$. 
Equivalently, a graph $G=(V,E)$ is  a comparability graph
if its edges admit a transitive orientation.
It is known~\cite[Theorem 7.1]{Golumbic}
that a graph $G$ is a permutation graph if and only if 
$G$ and $\bar{G}$ are comparability graphs.

In addition, in the proof of~\cite[Theorem 7.1]{Golumbic}, 
it is shown that, from each pair $(F,\bar{F})$ of
transitive orientations of $G$ and $\bar{G}$,
we can build a realizer $\pi$ of $G$.

Moreover, it easy to see that any realizer $\pi$ can be obtained in this way.
Indeed, let $G$ be a permutation graph, and $\pi$ be a realizer of $G$.
Then there is a labeling $v_1, ... ,v_n$ of the vertices of $G$ such that
$\{v_i, v_j\} \in V_G \Leftrightarrow \left(\pi(i)-\pi(j) \right)(i-j)<0$.
Hence $\{v_i, v_j\} \in V_{\bar{G}} \Leftrightarrow \left(\pi(i)-\pi(j) \right)(i-j)>0$.
We build the orientation $F$ (resp.~$\bar{F}$) by orienting
any edge $\{v_i,v_j\}$ in $G$ (resp. in~$\bar{G}$) from $v_i$ to $v_j$ if and only if $i<j$.
Then it is straightforward to check that $F$ (resp.~$\bar{F}$) is transitive and that this construction is the inverse of the one in the proof of \cite[Theorem 7.1]{Golumbic}.

Finally, 
it is known \cite[Corollary 5.13]{Golumbic} that a modular-prime comparability graph has exactly two transitive orientations, one being the inverse of the other. 

\begin{proof}[Proof of  \cref{prop:nb-antecedents-simple}]
Let $G$ be a  modular-prime permutation graph. Then $\bar{G}$ is also modular-prime.
By \cite[Corollary 5.13]{Golumbic}, both $G$  and $\bar{G}$ have exactly two transitive orientations.
Since realizers of $G$ are built from pairs $(F,\bar{F})$ of 
transitive orientations of $G$ and $\bar{G}$, the
graph $G$ has at most 4 realizers.
 \end{proof}
 
\begin{remark}
  We note that a modular-prime permutation graph may have less than 4 realizers
 since different pairs $(F,\bar{F})$ may yield the same realizer $\pi$ of $G$. This happens in fact when some/all realizers $\pi$ of $G$ has/have some dihedral symmetry.
\end{remark}

\subsection{Limit of a uniform permutation graph}
\label{ssec:mainproofperm}

In this section, we prove Theorem~\ref{Th:ConvUnifPermutationGraph}, which states that the sequence of uniform random permutation graphs $(\bm G_n)_n$ converges almost surely
to $\Wperm$  in the space of graphons.

\begin{proof}[Proof of Theorem~\ref{Th:ConvUnifPermutationGraph}]
  We denote by $\mathcal G^{\text{perm}}_n$ the set of permutation graphs with $n$ vertices.
  Let $\eps>0$. We have
  \[\mathbb P\big[ \delta_\Box(\bm G_n, \Wperm) \geq \eps \big]
  =\frac{\# \left\{ G \in \mathcal G^{\text{perm}}_n: \delta_\Box(G, \Wperm) \geq \eps \right\} }
  {\#\, \mathcal G^{\text{perm}}_n }.\]
  By definition, a permutation graph $G$ with $n$ vertices is equal to $G_\sigma$ for at least one permutation $\sigma$.
  Hence, the numerator can be bounded by
  \[ \# \left\{ G \in \mathcal G^{\text{perm}}_n: \delta_\Box(G, \Wperm) \geq \eps \right\} 
  \le \#  \left\{ \sigma \in S_n: \delta_\Box(G_\sigma, \Wperm) \geq \eps \right\} ,\]
  where $S_n$ denotes, as usual, the set of permutations of size 
  $n$.
  On the other hand using that modular-prime permutation graphs $G$ write as $G_\sigma$ for at most $4$ permutations,
  all simple (\cref{prop:nb-antecedents-simple}), we have
  \[\#\,  \mathcal G^{\text{perm}}_n 
  \ge \#  \big\{ G \in \mathcal G^{\text{perm}}_n : G \text{ modular-prime} \} \big\}
  \ge \frac{1}{4}\, \# \left\{ \sigma \in S_n: \sigma \text{ simple }   \right\}.\]
From~\cite{AAK03}, we know that, as $n \to \infty$, the number of simple permutations
is asymptotically $e^{-2}\, n!$. Thus, for $n$ large enough, using $4e^2<30$, we get
$\#\,  \mathcal G^{\text{perm}}_n               
\ge \frac{n!}{30}$.
Bringing everything together we have, for $n$ large enough
\begin{multline*}
  \mathbb P\big[ \delta_\Box(\bm G_n, \Wperm) \geq \eps \big]   
\le \frac{30}{n!}\, \# \big\{ \sigma \in S_n: \delta_\Box(G_\sigma, \Wperm) \geq \eps \big\}  \\
= 30 \, \mathbb P\big[  \delta_\Box(G_{\sigma_n}, \Wperm) \geq \eps \big],
\end{multline*}
where, in the last equation, $\sigma_n$ is a uniform random permutation of size $n$.
As in the proof of \cref{prop:cv_Gsigma}, we know that the upper bound in the above equation is summable as $n$ tends to $+\infty$.
Since this holds for any $\eps>0$, we have proved the theorem, again using the Borel--Cantelli Lemma.
\end{proof}

\subsection{Clique density in $\Wperm$}\label{ssec:nb-clique-perm}

\begin{proposition}\label{prop:MarginalesWinv}
Denote by $K_k$ the clique of size $k$. For every $k\geq 1$, 
$$
\dens{K_k}{\Wperm}=\frac{1}{k!}.
$$
\end{proposition}
Consequently, for every $k\geq 1$,
$$
\frac{1}{\binom{n}{k}}\ \#\set{\text{cliques of size $k$ in }\bm G_n }\ \to\ \frac{1}{k!}\ \text{ a.s.}.
$$
\begin{proof}
By definition, $\dens{K_k}{\Wperm}$ is the probability that $\Sample_k(\Wperm)$
is a clique $K_k$. Recall from \eqref{eq:LoiSample_Wperm} p.\pageref{eq:LoiSample_Wperm}
that  $\Sample_k(\Wperm)$ is distributed as 
the inversion graph $G_{\pi}$ of a uniform random permutation $\pi$ of size $k$. Moreover, $G_{\pi}=K_k$ if and only if $\pi$ is the decreasing permutation $d_k :=k \ (k\! -\! 1) \cdots 1$.
Summing up,
\[ \dens{K_k}{\Wperm} =\mathbb P\big(\Sample_k(\Wperm)=K_k\big)
=\mathbb P(G_{\pi}=K_k)= \mathbb P(\pi=d_k) = \frac1{k!}.
\qedhere \]
\end{proof}

\section{Circle graphs}
\label{sec:circle}

The main goal of this section is to prove \cref{Th:ConvUnifCircleGraph}: a sequence of uniform random circle graphs converges to the graphon $\Wcircle$ defined below in \cref{def:Wcircle} and \cref{prop:gcross}. 

The strategy of the proof is similar to the one used in the previous section for permutation graphs. 
We start by observing that the intersection graph of a uniform random matching (which is a \emph{non}uniform circle graph) converges to $\Wcircle$ (see \cref{prop:gcross}). 
Then, in order to transfer this result to uniform circle graphs, we will go through \emph{split-prime} circle graphs (a notion reviewed in \cref{ssec:indec_matching} below). Indeed, one can control the number of realizers of a split-prime circle graph (see \cref{coro:1a4n-matching-pr-premier}). 

A noticeable difference with the previous section comes from the enumeration of combinatorial objects corresponding to a prime graph. 
Indeed, in our proof for permutation graphs, we used previously known results on the enumeration of simple permutations (which correspond to modular-prime permutation graphs). 
However, for circle graphs, we need to define (in \cref{def:decomposable_matching} below) 
the analogous notion of \emph{indecomposable} matchings 
(which correspond to split-prime circle graphs), 
and then to estimate the number of these indecomposable matchings.
This additional step of the proof is dealt with in \cref{ssec:indec_matching}. 

\subsection{Limit of the intersection graph of a uniform matching}
We need to introduce some combinatorial objects and a bit of notation.

We define a matching of size $n$ as a fixed-point free involution on the set  $[2n]$. 
(They are sometimes called \emph{perfect matchings} or \emph{chord diagrams} in the literature.) 
Denote by $\mathcal{C}$ the unit circle centered at the origin, and set $\omega_{n}=e^{2\mathbf{i}\pi/(2n)}$.
For $\match$ a matching of size $n$, the circular representation of $\match$, denoted $\Circle (\match)$, is the chord configuration of $\mathcal{C}$ in which we put the $n$ chords\footnote{Chords are intended as straight lines, but for better readability we draw them curvy in pictures, being careful not to introduce unnecessary intersections.} of the form
$$
(\omega_n)^i \leftrightarrow (\omega_n)^{\match(i)}.
$$
By abuse of notation, we often identify a matching and its circular representation. The size of a matching is then its number of chords. 

As explained in the introduction, a graph $G$ with $n$ vertices is a \emph{circle graph} if $G$ is the (unlabeled) intersection graph 
 of $\Circle(\match)$ for a certain matching $\match$ of size $n$.
 Then $\match$ is called a \emph{realizer} of $G$
 and we write $G=G_\match$.
 
For any $n\geq 1$ let $\mathcal{M}_n$ be the set of matchings of size $n$. It will be useful for later purposes to observe that
\begin{equation}\label{eq_nb_matchings}
m_n:=\#\mathcal{M}_n = (2n-1)!! = \frac{(2n)!}{2^{n}\,n!}\stackrel{n\to +\infty}{\sim} \sqrt{2}(2n/e)^n.
\end{equation}

\medskip

Let $M_n$ be a uniform element in $\mathcal{M}_n$.
In this subsection we compute the graphon limit of $G_{M_n}$.
We first need the notion of crossing of two pairs of reals.
Let $x_A,x_B,y_A,y_B$ be four reals in $[0,1]$, identified to the points $e^{2\mathbf{i}\pi x_A}$, $e^{2\mathbf{i}\pi x_B}$, $e^{2\mathbf{i}\pi y_A}$ and $e^{2\mathbf{i}\pi y_B}$ on the unit circle. We say that $(x_A,x_B)$ and $(y_A,y_B)$ are \emph{crossing} if $x_A,x_B,y_A,y_B$ are pairwise distinct and if the chords $e^{2\mathbf{i}\pi x_A}\leftrightarrow e^{2\mathbf{i}\pi x_B}$ and $e^{2\mathbf{i}\pi y_A}\leftrightarrow  e^{2\mathbf{i}\pi y_B}$ intersect (\emph{i.e.}~$x$'s and $y$'s alternate in the circular order).

\begin{definition}
\label{def:Wcircle}
Let
$$
\begin{array}{r c c c}
\Psi: & [0,1] & \to & [0,1]^2\\
        & x   & \mapsto & (\Psi_A(x),\Psi_B(x))
\end{array}
$$
be any measurable function which is Lebesgue-preserving. The graphon $\Wcircle_\Psi$ is defined by (the equivalence class of)
$$
\Wcircle_\Psi(x,y)=\mathbb{1}\left[(\Psi_A(x),\Psi_B(x)) \text{ and }(\Psi_A(y),\Psi_B(y))\text{ are crossing}\right] \quad \text{ for all }x,y\in [0,1].
$$
\end{definition}
As for $\Wperm$, we can avoid the use of a Lebesgue-preserving function $\Psi$
by using a more general formalism for graphons, see \cref{rmk:graphons-on-general-S}. 

\begin{proposition}\label{prop:gcross}
The equivalence class of $\Wcircle_\Psi$ is independent
of the choice of the Lebesgue-preserving function $\Psi$. 
Moreover, let us consider, for each $n$, 
a uniform random matching $M_n$ of size $n$.
Then its  intersection graph $G_{M_n}$ satisfies
$$
G_{M_n} \stackrel{\text{a.s.}}{\to} \Wcircle,
$$
where $\Wcircle:=\Wcircle_\Psi$ for an arbitrary Lebesgue-preserving function $\Psi$.
\end{proposition}

\begin{proof}
As in the proof of \cref{prop:cv_Gsigma} we first identify the distribution of 
$\mathrm{Sample}_n(\Wcircle_\Psi)$ (for $n \ge 1$).
This random graph on vertex set $[n]$ is constructed by taking i.i.d.~uniform random variables $X_1$, \dots, $X_n$ in $[0,1]$
and by connecting vertices $i$ and $j$ if and only if the chords $(\Psi_A(X_i),\Psi_B(X_i))$ and $(\Psi_A(X_j),\Psi_B(X_j))$ are crossing.

For all $i$ in $[n]$, let $(u_{2i-1},u_{2i})=\Psi(X_i)$. Since $\Psi$ is Lebesgue-preserving,
the $2n$ numbers $u_1,\dots,u_{2n}$ are i.i.d.~uniform in $[0,1]$. Let $\tau$ be the unique permutation on $[2n]$ such that
$$
u_{\tau(1)}<\dots< u_{\tau(2n)}.
$$
Then, $\mathrm{Sample}_n(\Wcircle_\Psi)$ is the intersection graph $G_{\mathfrak m}$ of the matching $\match$ of size $n$ defined by
$$
\match=(\tau^{-1}(1),\tau^{-1}(2))(\tau^{-1}(3),\tau^{-1}(4))\dots (\tau^{-1}(2n-1),\tau^{-1}(2n)).
$$
Moreover, the permutation $\tau$ is a uniform permutation of size $2n$ (and so is $\tau^{-1}$). 
So, the matching $\match$ is a uniform matching of size $n$, implying
\begin{equation}\label{eq:LoiSample_Wcircle}
\mathrm{Sample}_n(\Wcircle_\Psi) =G_{\match}
\stackrel{\text{(d)}}{=}G_{M_n}.
\end{equation}
We conclude as in the proof of \cref{prop:cv_Gsigma}: for every fixed $\eps>0$ 
\begin{equation}\label{eq:G_M_n_sommable}
\mathbb{P}\big[\delta_{\Box}(G_{M_n},\Wcircle_\Psi) \ge \eps\big]=\mathbb{P}\big[\delta_{\Box}(\mathrm{Sample}_n(\Wcircle_\Psi),\Wcircle_\Psi) \ge \eps \big],
\end{equation}
which is summable by \cref{lem:EstimeeSample}.
The Borel--Cantelli Lemma yields $G_{M_n}\stackrel{\text{a.s.}}{\to} \Wcircle_\Psi$.
\end{proof}

\subsection{(In)decomposability of matchings}\label{ssec:indec_matching}

\subsubsection{Indecomposable matchings}

We first define indecomposable matchings, which will be an analog of simple permutations for matchings.
They enjoy the nice property that split-prime circle graphs (whose definition is reviewed below) are represented by indecomposable matchings  (see \cref{prop:prime=indec}).

\begin{definition}
\label{def:decomposable_matching}
Let $\match$ be a matching of size $n$. We say that $\match$ is {\em $k$-decomposable} if
there exists a partition of $[2n]$ into four (possibly empty) parts $C_1, C_2, C_3, C_4$ such that
\begin{itemize}
\item  each $C_i$ is a circular interval (\emph{i.e.} an interval of $\{1, 2, \dots, 2n\} \mod 2n$);
\item  $C_1$ contains $1$, and the nonempty parts among $C_2$, $C_3$ and $C_4$ are ordered according to their smallest element;
\item all chords have either both extremities in $C_1 \cup C_3$ or both extremities in $C_2 \cup C_4$;
\item $C_2 \cup C_4$ contains exactly $k$ chords.
\end{itemize}
A matching of size $n$ is {\em decomposable} if it is $k$-decomposable for some $k$ with $2 \le k\le n-2$.\\
A matching is {\em indecomposable} if it is not decomposable.
\end{definition}
Observe that a matching can be $k$-decomposable for several $k$. An example of decomposable matching is given in \cref{fig:k-dec-matching}~(left). 

\begin{remark}\ 
\begin{itemize}
\item If $\match$ is decomposable then, as indicated above, $C_i$ may be empty for some $i \in \{2,3,4\}$. But then, from the last item of \cref{def:decomposable_matching} and the bounds on $k$, for $j \equiv i+2 \mod 4$,  $C_j$ should contain at least four points of the matchings (two chords).
\item We warn the reader that other notions of indecomposable matchings have appeared
in the literature see, e.g., \cite{Jefferson2015substitution-matchings}.
In the latter reference, two notions of weakly and strongly indecomposable matchings
are considered, both being weaker than the one considered here.
\end{itemize}
\end{remark}

\begin{figure}[h!]
\begin{center}
\includegraphics[width=11cm]{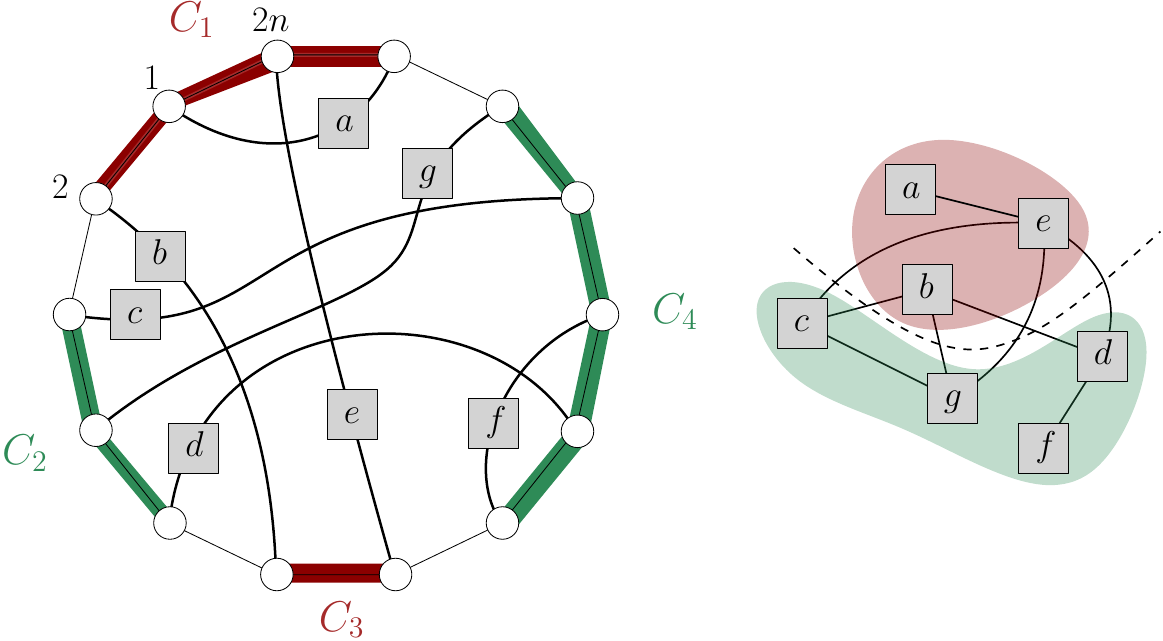}
\end{center}
\caption{Left: The circular representation of a $k$-decomposable matching for $k=\#\set{c,d,f,g}=4$. Right: the corresponding intersection graph with the cut $\set{a,b,e}\uplus\set{c,d,f,g}$ induced by $C_1,\dots ,C_4$. The corresponding cut-set  $\set{b,e}\times\set{c,d,g}$ is a complete bipartite graph and the split is depicted with the dashed line.}\label{fig:k-dec-matching}
\end{figure}

\medskip

We now introduce the necessary terminology to relate indecomposable matchings to split-prime circle graphs. 

Recall that a \emph{cut} of a graph is a partition of the vertices into two nonempty subsets $V_1$ and $V_2$, called the sides of the cut. The subset of edges that have one endpoint in each side of the cut is called a \emph{cut-set}, and a cut whose cut-set forms a (possibly empty) complete bipartite graph is called a \emph{split}.
By extension, the two sets of vertices in the complete bipartite graph defining a split will be refered to as cut vertex sets (this shall not be confused with the notion of {\em cut vertices},
not relevant here).
An equivalent definition is to say that a cut $V_1$ and $V_2$ form a split if and only if they contain subsets $V_1^{cut}$ and $V_2^{cut}$ (possibly empty) such that:
\begin{itemize}
  \item there is no edge between $V_1 \setminus V_1^{cut}$ and $V_2$, and similarly no edge between  $V_2 \setminus V_2^{cut}$ and $V_1$;
  \item any vertex in $V_1^{cut}$ is linked to any vertex in $V_2^{cut}$.
\end{itemize}

A split is trivial when one of its two sides has only one vertex in it. A graph is said to be prime for the split decomposition, or \emph{split-prime} for short, if it has no nontrivial splits ; otherwise it is \emph{split-decomposable}.


\begin{proposition}
\label{prop:prime=indec}
Let $G$ be a circle graph and $\match$ be a matching that represents $G$.
Then $G$ is split-prime if and only if $\match$ is indecomposable.
\end{proposition}

\begin{proof}
The proof is technical and postponed to \cref{Sec:AppendiceMatching}.
\end{proof}

\subsubsection{Enumeration of decomposed matchings}

Our goal is to prove that a positive proportion of matchings are indecomposable (in \cref{lem:nb-indecomposable} below).
To this end we define
a {\em $k$-decomposed} matching as a pair formed by a $k$-decomposable matching $\match$ and a decomposition $(C_1, C_2, C_3, C_4)$ of $\match$ as in \cref{def:decomposable_matching}.
We first study the number of $k$-decomposed matchings.

\begin{lemma}\label{lem:nb_decomposed_matching}
For any $2 \leq k \leq n-2$, let $d_n^k$ be the number of $k$-decomposed matchings of size $n$.
Then 
$$
d_n^k = (n-k)\,m_{k+1}\,m_{n-k+1},
$$
 where $m_n$ is the number of matchings of size $n$.
\end{lemma}

\begin{proof}
Let $\mathcal{M}^\bullet_k$ be the set of matchings of size $k$ with a marked chord such that the marked chord is not the one containing $1$. Then $\#\mathcal{M}^\bullet_k = (k-1)\,m_k$.
We prove the lemma by giving a one-to-one correspondence $\varphi$ between the set $\mathcal{D}_n^k$  of $k$-decomposed matchings of size $n$ and $\mathcal{M}^\bullet_{n-k+1} \times \mathcal{M}_{k+1}$.
This construction is illustrated on \cref{Fig:Bijection_matching}. 

\begin{figure}[h]
\begin{center}
\includegraphics[width=8cm]{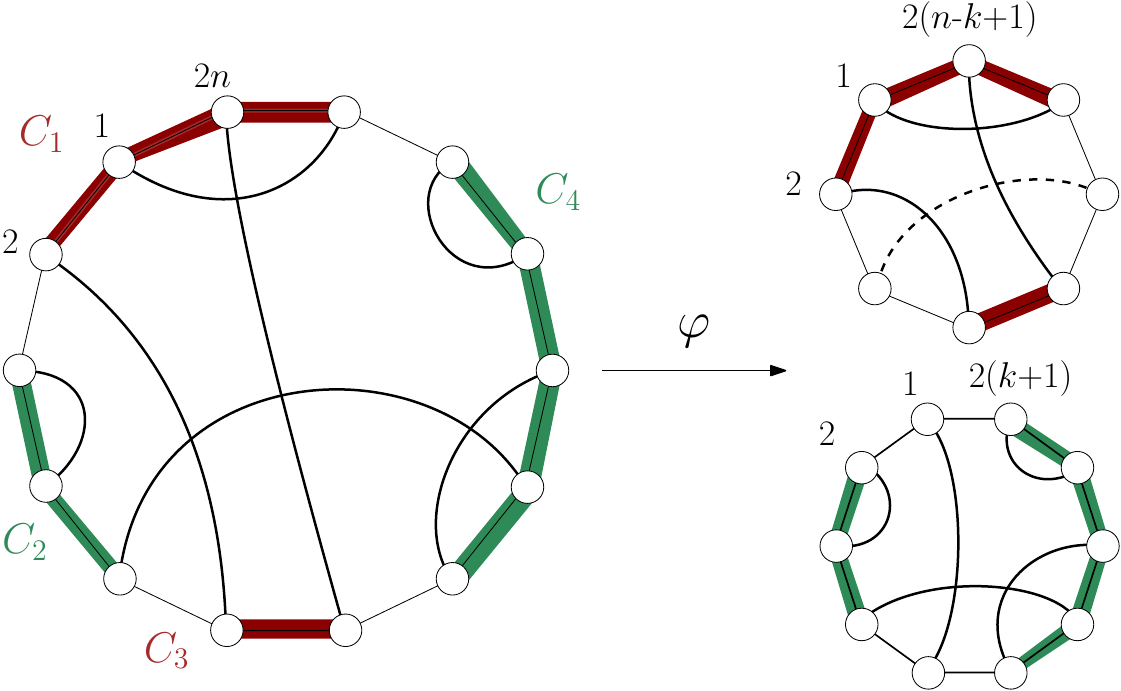}
\end{center}
\caption{Notation of the proof of \cref{lem:nb_decomposed_matching}. Left: A $k$-decomposed matching. Right: The corresponding matchings $\match$ of size $n-k+1$ (marked at the dashed chord) and $\match'$ of size $k+1$.}
\label{Fig:Bijection_matching}
\end{figure}

Let $(C_1, C_2, C_3, C_4) \in \mathcal{D}_n^k$. To obtain a matching $\match$ in $\mathcal{M}^\bullet_{n-k+1}$, we glue $C_1$ and $C_3$ together and add a marked chord separating the two parts (to remember where it has been cut; in \cref{Fig:Bijection_matching} this chord is depicted by a dashed line).  
Then we relabel the points on the circle so that the $1$ of $C_1$ remains $1$ in $\match$, and the other points are labeled by increasing order starting from $1$.
Similarily, to obtain a matching $\match'$ in $\mathcal{M}_{k+1}$, we glue $C_2$ and $C_4$ together and add a chord separating the two parts.
To remember the added chord, we label by $1$ the point of the added chord in $\match'$ that is next to the smallest number of $C_2$, and the remaining points are labeled by increasing order starting from $1$. 
If $C_2$ is empty then extremities of the added chord are labeled $1,2$, and the remaining points are also labeled in increasing order. 
We set $\varphi(C_1, C_2, C_3, C_4)=(\match,\match')$. 

To see that $\varphi$ is a bijection, we construct its inverse.  Let $\match,\match'$ be a pair of matchings in $\mathcal{M}^\bullet_{n-k+1} \times \mathcal{M}_{k+1}$. 
We cut $\match$ along the marked chord (hence deleting this chord).
This gives two circular arcs: we call $C_1$ the one containing $1$ and $C_3$ the other one.
Then we cut $\match'$ along the chord containing $1$ (this chord is also deleted).
This gives two circular arcs. We call $C_2$ the one containing the $2$ of $\match'$ (unless $\match'$ contains a chord from $1$ to $2$, in which case $C_2$ is empty). We call $C_4$ the other one.
We build a $k$-decomposed matching of size $n$ by gluing these four circular arcs, in increasing order of their indices, and preserving the orientations of the circular arcs. 
Finally we label the points in increasing order so that the $1$ of the new matching is the $1$ of $\match$. 
It should be clear that applying $\varphi$ to the four circular arcs defined above gives $(\match,\match')$. 
\end{proof}

In the sequel we will estimate $d_n^k$ for every $2\leq k\leq n-2$. As we will see, the case $k=2$ is quite different from $k>2$.

\subsubsection{Probability of $k$-decomposability for $k>2$}

Recall that $M_n$ is a random matching  taken uniformly at random among the $m_n$ matchings of size $n$. 

\begin{lemma} As $n$ tends to $+\infty$, we have
\label{lem:not-k-decomposable}
\begin{equation}\label{eq:not-k-decomposable}
\sum_{k=3}^{n-2}\frac{d_n^k}{m_n} = \BigO \left(\frac1n\right).
\end{equation}

Consequently, 
$$
 \mathbb{P}(M_n\text{ is $k$-decomposable for some }k>2) \to 0.
$$
\end{lemma}

\begin{proof}
Recall that for all $n$, we have $m_n = (2n-1)m_{n-1}$,
and that, from \cref{lem:nb_decomposed_matching}, we have
$d_n^k = (n-k)\,m_{k+1}\,m_{n-k+1}$.
We set $e_n^k =m_{k+1}\,m_{n-k+1}$.
Trivially,
\[\frac{e^{k+1}_n}{e^k_n} = \frac{m_{k+2}\,m_{n-k}}{m_{k+1}\,m_{n-k+1}} = \frac{2k+3}{2n-2k+1}.\]
Hence, for fixed $n$, the map $k \mapsto {e^{k}_n}$ is decreasing for $k \le \frac{n-1}2$,
and increasing for $k \ge \frac{n-1}2$.
In particular, using $d_n^k \le n e_n^k$ after isolating some terms in the sum,
we have the bound
\begin{equation}\label{eq:intermediate_bound_proportion_dec}
  \sum_{k=3}^{n-2}\frac{d_n^k}{m_n} \le \frac{d_n^3 + d_n^{n-3}+d_n^{n-2}}{m_n} 
+ n \, \sum_{k=4}^{n-4} \frac{e_n^k}{m_n}
\le \frac{d_n^3 + d_n^{n-3}+d_n^{n-2}}{m_n} +n^2 \frac{\max(e_n^4,e_n^{n-4})}{m_n}.
\end{equation}
Trivially, for fixed $k$,
\[ \frac{e_n^k}{m_n}=\frac{e_n^{n-k}}{m_n}=
\frac{m_{k+1}}{(2n-1)\dots(2n-2k+3)}=O(n^{-k+1}),\]
implying
\[\frac{d_n^k}{m_n}=(n-k)\, \frac{e_n^k}{m_n}=O(n^{-k+2})
\text{ and } \frac{d_n^{n-k}}{m_n}=k \frac{e_n^k}{m_n}=O(n^{-k+1}). \]
Using these bounds in \eqref{eq:intermediate_bound_proportion_dec}
ends the proof of \eqref{eq:not-k-decomposable}.
The probabilistic consequence is immediate, since the number of $k$-decomposable matchings
is at most that of $k$-decomposed matchings.
\end{proof}

\subsubsection{Probability of $2$-decomposability}
\label{ssec:2decomposability}

\begin{lemma}\label{lem:not-2-decomposable}
As above, let $M_n$ be a uniform random matching of size $n$. We have
$$
 \mathbb{P}(M_n\text{ is not $2$-decomposable}) \stackrel{n\to\infty}{\to} e^{-3}.
$$
\end{lemma}

The remainder of \cref{ssec:2decomposability} is devoted to the proof
of \cref{lem:not-2-decomposable}. 
To this end, we introduce the following quantities.
For a matching $\match$, let $x(\match) = \sum_{i=1}^{2n} \One[\match(i)\equiv i+1]$ be the number of chords between adjacent points, where $\equiv$ stands for equality mod $2n$. Similarly we set $y(\match)=\sum_{j=1}^{2n} \One[\match(j)\equiv j+2]$. Finally let 
$$
z(\match)=\sum_{\substack{1\leq k<\ell\leq 2n \\ \ell-k\not\equiv \pm 1}} \One\big[\set{\match(k),\match(k+1)}\equiv\set{\ell,\ell+1}\big];
$$
\emph{i.e.}~$z(\match)$ counts pairs of consecutive points matched to another pair of consecutive points.
 The definitions of $x(\match)$, $y(\match)$ and $z(\match)$ are illustrated on \cref{Fig:XnYnZn}.

\begin{figure}[ht]
\begin{center}
\includegraphics[width=11cm]{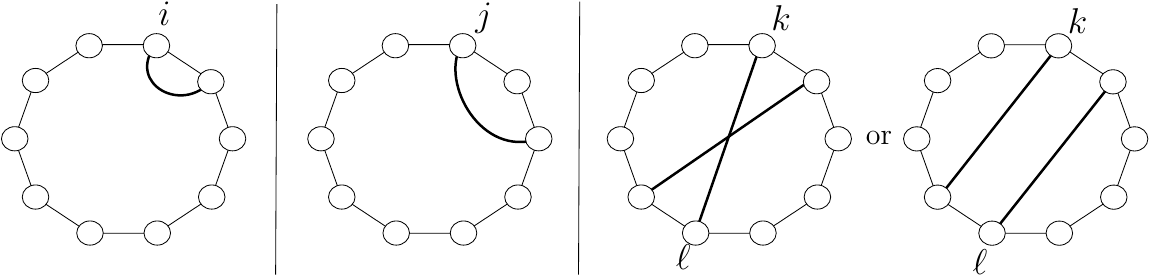}
\end{center}
\caption{Sub-configurations of a matching $\match$ counted by $x(\match)$, $y(\match)$, $z(\match)$.}
\label{Fig:XnYnZn}
\end{figure}

\begin{lemma}
\label{lem:XYZ-decomposability}
For a matching $\match$, whose size we denote $n$, we have $x(\match)=y(\match)=z(\match)=0$ if and only if $\match$ is neither $2$-decomposable nor $(n-2)$-decomposable. 
\end{lemma}

\begin{proof}
We shall prove that if one of $x(\match)$, $y(\match)$ or $z(\match)$ is positive,
then $\match$ is either $2$-decomposable or $(n-2)$-decomposable. 
The proof of the converse statement is easy and left to the reader.
\begin{itemize}
\item Assume first that $z(\match)>0$, \emph{i.e.}~there exists $k<\ell$ such that
$$\set{\match(k),\match(k+1)}\equiv\set{\ell,\ell+1}.$$ 
Then the set $[2n]$ can be split into four (possibly empty) circular intervals
\[\{k,k+1\}, \{k+2,\dots \ell-1\}, \{\ell,\ell+1\}, \{\ell+2,\dots,k-1\},\]
such that all chords have both extremities in the same interval or in diagonally facing intervals. Note that  $\{k,k+1\} \cup \{\ell,\ell+1\}$ contains exactly two chords.
The labeling of these circular intervals as $C_1$, $C_2$, $C_3$ and $C_4$
depends on whether $k=1$ or $\ell=n$, or none of those, so that either $C_1 \cup C_3$ or $C_2 \cup C_4$  contains 2 chords.
 Thus $z(\match)>0$ implies that the matching $\match$ is either $2$ or $(n-2)$-decomposable.
\item Assume now that $y(\match)>0$, \emph{i.e.}~there exists $j$ such that $\match(j)=j+2$.
 Letting $k=\match(j+1)$, the following partition
of $[2n]$ into four (possibly empty) circular intervals
\[\{j,j+1,j+2\}, \{j+3,\dots k-1\}, \{k\}, \{k+1,\dots,j-1\}\]
shows that $\match$ is either $2$ or $(n-2)$-decomposable.
\item The same conclusion holds true if $x(\match)>0$;
the proof is similar, except that we set $k=\match(i+2)$.\qedhere
\end{itemize}
\end{proof}

Recalling that $M_n$ denotes a uniform random matching of size $n$, 
let us define  
$X_n=x(M_n)$, $Y_n=y(M_n)$ and $Z_n=z(M_n)$. 

\begin{lemma}
\label{lem:XYZ-poisson}
The triple $(X_n,Y_n,Z_n)$ converges in distribution towards a triple of independent Poisson random variables with mean 1. 
\end{lemma}
\begin{proof}
The proof is technical and postponed to \cref{Sec:AppendicePoisson}.
\end{proof}

\begin{proof}[Proof of \cref{lem:not-2-decomposable}]
Using \cref{lem:not-k-decomposable} for $k=n-2$,
we have
\[\mathbb{P}(M_n\text{ is not $2$-decomposable})=
\mathbb{P}(M_n\text{ is neither $2$ nor $n-2$-decomposable})+\mathrm{o}(1). \]
But \cref{lem:XYZ-decomposability} asserts that the latter event is the same as
\enquote{$X_n=Y_n=Z_n=0$}. Therefore
\[\mathbb{P}(M_n\text{ is not $2$-decomposable})=
\mathbb{P}(X_n=Y_n=Z_n=0)+\mathrm{o}(1),\]
and the latter tends to $e^{-3}$ by \cref{lem:XYZ-poisson}.
   \end{proof}

\subsubsection{Probability of indecomposability}

Bringing together \cref{lem:not-k-decomposable,lem:not-2-decomposable}, we can give a lower bound on the probability for a uniform matching to be indecomposable.

\begin{proposition}\label{lem:nb-indecomposable}
Let $M_n$ be a uniform matching of size $n$. Then for $n$ large enough,
$$
 \mathbb{P}(M_n\text{ is indecomposable}) > e^{-4}.
$$
\end{proposition}

\begin{proof}
We write
\begin{multline*}
 \mathbb{P}(M_n\text{ is decomposable})
   \leq \mathbb{P}(M_n\text{ is $2$-decomposable}) \\
   + \mathbb{P}(M_n\text{ is $k$-decomposable for some }k>2).
   \end{multline*}
   From \cref{lem:not-k-decomposable,lem:not-2-decomposable},
   the latter converges to $1 - e^{-3}$ as $n$ tends to $+\infty$, implying the proposition.
\end{proof}

\subsection{On the number of matchings that represent a given circle graph}

Let $\match$ be a matching. Then the {\em shift} of $\match$ is the matching obtained by replacing each chord $(i,j)$ of $\match$ by the chord $(i+1,j+1)$ (where $2n+1$ is identified with $1$).
This operation corresponds to a rotation of the circular representation of $\match$.
Moreover, the {\em reversal} of $\match$ is the matching obtained by replacing each chord $(i,j)$ of $\match$ by the chord $(2n+1-i,2n+1-j)$.
This operation corresponds to a symmetry of the circular representation of $\match$ (specifically, to the reflexion with respect to the diameter passing between points labeled $1$ and $2n$ on one side, and between points labeled $n$ and $n+1$ on the other side).

Let $G$ be a circle graph and $\match$ be a matching that represents $G$.
Then every matching obtained from $\match$ by a sequence of shifts and reversals also represents $G$.
It has been proved that when $G$ is split-prime, this is the unique source for the lack of uniqueness of the representative:

\begin{proposition}[Corollary in Section 8 of \cite{Gabor}]\label{prop:Gabor-unicite}
Let $G$ be a split-prime circle graph with at least five vertices.
Then there is a unique (up to shifts and reversals) matching $\match$ such that $G=G_\match$.
\end{proposition}

\begin{corollary}\label{coro:1a4n-matching-pr-premier}
Let $G$ be a split-prime circle graph with $n \geq 5$ vertices. The number of matchings $\match$ such that $G=G_\match$ is between $1$ and $4n$.
\end{corollary}

\begin{proposition}\label{prop:<n-representants}
The proportion of matchings of size $n$ whose associated circle graphs
 have strictly less than $4n$ representatives is $o(n^{-n/3})$.
\end{proposition}
As we will see in the proof, this bound is far from optimal but sufficient for our purposes.
\begin{proof}
The circle graph associated to a matching of the numbers $\{1,\dots,2n\}$
may have less than $4n$ representatives only if 
the matching is fixed by some nontrivial symmetry $s$ of a regular $2n$-gon.

Let us fix some symmetry $s$ of a regular $2n$-gon. 
We consider the set $\mathcal M^s_{n}$ of matchings fixed by $s$,
and denote by $m^s_n$ its cardinality. 
The cases where $s$ is a rotation of order $2$
and where $s$ is a reflection
 lead to the same enumeration sequence 
 -- too see this, simply switch $2n-i$ and $n+i$ for $i \le n/2$; 
 this turns a rotation-invariant matching into a reflection-invariant one.
Hence we assume that $s$ is a rotation. Denote its order by $d$.
We have $2n=dk$ for some integer $k$.
An element $M$ in $\mathcal M^s_{n}$ is uniquely encoded by a partition $\pi$
of the numbers $\{1,\dots,k\}$ into singletons and pairs
 with the following constraints and interpretation.
\begin{itemize}
\item Singletons are allowed only if $d$ is even (\emph{i.e.}~when $n$ is a multiple of $k$);
a singleton $\{i\}$ represent the chord $\{i,i+n\}$ and its rotations ($\{i+k,i+k+n\}$, \dots, 
$\{i+n-k,i+2n-k\}$). Having $d$ even and containing these chords is the only possibility for a rotation-invariant
matching $M$ to contain a chord joining $i$ to some integer in the same class mod.~$k$;
hence singletons in $\pi$ encode all chords of $M$ joining integers within the same modulo class.
\item Pairs $\{i,j\}$ are decorated with a number $h$ in $\{0,\dots,d-1\}$
and represent the chord $\{i,j+h\, k\}$ and its rotations ($\{i+k,j+(h+1)\, k\}$, \dots, $\{i+(d-1)k,j+(h+d-1)\, k\}$, working mod.~$2n$).
These pairs in $\pi$ encode all chords of $M$ joining integers in different classes modulo $k$.
\end{itemize}
Using the formalism of labeled combinatorial classes \cite[Chapter II]{Violet},
this yields the following expression for an exponential generating series
of $\mathcal M^s_{n}$. For fixed $d\ge 2$,
\[ 
\sum_{k \ge 1} m^s_{dk/2} \frac{z^{k}}{k!} = \exp\big( z \delta_{d\text{ even}} + d \tfrac{z^2}2\big).
\]
With Cauchy formula, we get
\[ m^s_{dk/2} = \frac{k!}{2 \pi\, i} \oint \frac{\exp\big( z \delta_{d\text{ even}} + d \tfrac{z^2}2\big)}{z^{k+1}} dz, \]
where we integrate over any counterclockwise contour around $0$.
We choose this contour to be the circle $\{z: |z|=\sqrt{k/d}\}$.
Recalling that $k^{k}/e^{k-1} \le k! \le k^{k+1}/e^{k-1}$ for all $k \ge 1$, we get the following upper bound:
\[m_n^s=m^s_{dk/2}  \le k!  \sqrt{k/d} \, \frac{ \exp( \sqrt{k/d} + \tfrac{k}{2}) }{\sqrt{k/d}^{k+1}}\\
\le k \, (kd)^{k/2} \exp\big( \sqrt {k/d} - k/2 + 1 \big).
\] 
For $d,k \ge 2$, we have $\exp\big( \sqrt {k/d} - k/2) \le 1$.
 Recalling that $2n=dk$, we have in particular
$m_n^s \le e (2n)^{\frac{n}{d}+1}$.
For unconstrained matchings, we have
\[ m_n=\frac{(2n)!}{2^n n!} \ge \frac{(2n)^{2n} \, e^{n-1}}{e^{2n-1} \, 2^n\, n^{n+1}} = \frac{2^n\,  n^{n-1}}{e^n}.\]
Comparing both bounds yields that for $n$ large enough, for any symmetry $s$
of order $d \ge 2$, we have 
$\frac{m_n^s}{m_n} =o(n^{-n/3-1})$, uniformly on $s$.
Since there are $4n$ possible symmetries $s$, this proves the proposition.
\end{proof}

\subsection{Proof of \cref{Th:ConvUnifCircleGraph}: limit of a uniform circle graph}

In this section, we prove \cref{Th:ConvUnifCircleGraph} which states that
the sequence of uniform random circle graphs $(\bm{G}_n)_n$ converges to $\Wcircle$ almost surely in the space of graphons.
\begin{proof}
  We denote by $\mathcal G^{\text{circle}}_n$ the set of circle graphs with $n$ vertices.
  Let $\eps>0$. We have
  \begin{equation}
  \label{eq:PGeps_circle}
  \mathbb P\big[ \delta_\Box(\bm{G}_n, \Wcircle) \geq \eps \big]
  =\frac{\# \left\{ G \in \mathcal G^{\text{circle}}_n: \delta_\Box(G, \Wcircle) \geq \eps \right\} }
  {\#\ \mathcal G^{\text{circle}}_n }\ .
  \end{equation} 
  
  By definition, a circle graph $G$ with $n$ vertices is equal to $G_\match$ for at least one matching $\match$.
  We denote by $\mathcal G^{<4n}_n$ (resp. $\mathcal G^{\geq4n}_n$) the set of circle graphs with $n$ vertices that have less than $4n$ representatives (resp. at least $4n$ representatives).
  Moreover, we denote by $\mathcal M^{<4n}_n$ (resp. $\mathcal M^{\geq4n}_n$) the set of matchings $\match$ such that $G_\match \in \mathcal G^{<4n}_n$ (resp. $G_\match \in \mathcal G^{\geq4n}_n$).
  Then,
  \begin{align}
  \begin{split}
  &\# \left\{ G \in \mathcal G^{\text{circle}}_n: \delta_\Box(G, \Wcircle) \geq \eps \right\} \\
  &= \# \left\{ G \in \mathcal G^{\geq4n}_n: \delta_\Box(G, \Wcircle) \geq \eps \right\} 
   + \# \left\{ G \in \mathcal G^{<4n}_n: \delta_\Box(G, \Wcircle) \geq \eps \right\} \\
  &\le \frac{1}{4n} \# \left\{ \match \in \mathcal M^{\geq4n}_n: \delta_\Box(G_\match, \Wcircle) \geq \eps \right\} + \# \left\{ \match \in \mathcal M^{<4n}_n: \delta_\Box(G_\match, \Wcircle) \geq \eps \right\} \\
  &\le \frac{1}{4n} \# \left\{ \match \in \mathcal M_n: \delta_\Box(G_\match, \Wcircle) \geq \eps \right\} + \# \left\{ \match \in \mathcal M^{<4n}_n\right\}.
   \end{split} \label{eq:upper_far_circle}
  \end{align} 

  On the other hand, using that split-prime circle graphs $G$ correspond to indecomposable matchings (\cref{prop:prime=indec})
  and that each split-prime graph is represented by at most $4n$ indecomposable matchings
  (\cref{coro:1a4n-matching-pr-premier}), we have
  \[ \#\,  \mathcal G^{\text{circle}}_n 
  \ge \#  \left\{ G \in \mathcal G^{\text{circle}}_n : G \text{ split-prime} \right\}
  \ge \frac{1}{4n}\, \#  \left\{ \match \in \mathcal{M}_n: \match \text{ indecomposable}   \right\}\]
From~\cref{lem:nb-indecomposable}, we know that, for $n$ large enough, the number of indecomposable matchings of size $n$
is asymptotically greater than $e^{-4}m_n$ where $m_n$ is the number of matchings of size~$n$. Thus, for $n$ large enough,
\begin{equation}
\label{eq:lower_circle}
\#\,  \mathcal G^{\text{circle}}_n                 
\ge \frac{e^{-4}m_n}{4n} .
\end{equation}
Bringing \eqref{eq:PGeps_circle}, \eqref{eq:upper_far_circle}
and~\eqref{eq:lower_circle} together, we have

\begin{align*}
  \mathbb P\big( \delta_\Box(\bm{G}_n, \Wcircle) \geq \eps \big)   
&\le \frac{e^{4}}{m_n}\, \# \left\{ \match \in \mathcal{M}_n: \delta_\Box(G_\match, \Wcircle) \geq \eps \right\}  +  4e^4 \, n \, \frac{\# \left\{ \match \in \mathcal M^{<4n}_n\right\}}{m_n}\\
&\le e^{4} \mathbb P\big( \delta_\Box(G_{M_n}, \Wcircle) \geq \eps \big)+4e^4 \, n \, \frac{\# \left\{ \match \in \mathcal M^{<4n}_n\right\}}{m_n}.
\end{align*}
We saw in the proof of \cref{prop:gcross} (see in particular \eqref{eq:G_M_n_sommable} and \cref{lem:EstimeeSample}) that the first term in the right-hand side is summable.
\cref{prop:<n-representants} tells that the second term is a $\mathrm{o}(n\times n^{-n/3})$.
Using the Borel--Cantelli Lemma (as in the proof of \cref{lem:EstimeeSample}), this concludes the proof of the theorem.
\end{proof}

\subsection{Clique density in $\Wcircle$}

\begin{proposition}\label{prop:MarginalesWmatch}
Denote by $K_k$ the clique of size $k$. For every $k\geq 1$ 
$$
\dens{K_k}{\Wcircle}=\frac{1}{m_k} = \frac{2^{k}k!}{(2k)!}.
$$
In particular the density of edges  $\dens{K_2}{\Wcircle}$ equals $1/3$ and the density of triangles $\dens{K_3}{\Wcircle}$ equals $1/15$.
\end{proposition}

\begin{proof}
By definition, $\dens{K_k}{\Wcircle}=\mathbb P[\Sample_k(\Wcircle)=K_k]$.
From \cref{eq:LoiSample_Wcircle}, we know that $\Sample_k(\Wcircle)$ is distributed as $G_{M_k}$,
where $M_k$ is a uniform random matching of size $k$.
It is easy to see that the only matching ${\mathfrak m}$ such that $G_{\mathfrak m}=K_k$
is the matching $\mathfrak m_0=\{\{1,k+1\},\{2,k+2\},\dots,\{k,2k\}\}$.
Hence,
\[\dens{K_k}{\Wcircle}=\mathbb P[\Sample_k(\Wcircle)=K_k]=
\mathbb P[M_k=\mathfrak m_0] = \frac{1}{m_k}.\]
We conclude using Equation~\eqref{eq_nb_matchings}.
\end{proof}

\section{Unit interval graphs}
\label{sec:unit-interval}

The main goal of this section is to prove \cref{thm:unit-interval}:
a sequence of uniform random unit interval graphs with a renormalized distance function converges in the sense of the Gromov--Prokhorov topology
towards the unit interval $[0,1]$, endowed with a random metric defined through a Brownian excursion.

An important difference with the previous sections is that the convergence is in the Gromov--Prokhorov topology and not in the sense of graphons. 
Nevertheless, we similarly focus a large part of our study on indecomposable combinatorial objects (irreducible Dyck paths here), which represent by an essentially unique way \emph{connected} unit interval graphs (playing the role of modular-prime permutation graphs or split-prime circle graphs). 
In this section, and unlike in the previous ones, our intermediate statement consists in establishing a limit result for the graph associated to a uniform \emph{indecomposable} combinatorial object, while we proved such results for uniform combinatorial objects in previous sections. 

We start by observing in \cref{sec:def_unit_interval} how connected unit interval graphs can be encoded by irreducible Dyck paths.
Then in \cref{ssec:limit-uniform-Dyck-path} we prove that the unit interval graph obtained from a uniform random irreducible Dyck path
converges in the sense of the Gromov--Prokhorov topology towards the unit interval $[0,1]$,
endowed with a random metric defined through a Brownian excursion
(the proof of a technical lemma is postponed to \cref{ssec:preuve-lemme-distance}).
In \cref{ssec:interval_removing_irred_condition} we transfer this result to uniform circle graphs.
Finally, an asymptotic result for the number of cliques of size $k$ ($k$ fixed) in a uniform random unit interval graph is given in \cref{ssec:nb-clique-unit}.

\subsection{Combinatorial encoding of unit interval graphs}
\label{sec:def_unit_interval}

An (unlabeled) graph $G$ is a unit interval graph if there exists a collection $\mathcal I=(I_1,...,I_n)$ of intervals of $\RR$ with {\em unit length} such that a labeled version $G^\ell$ of $G$ is the intersection graph associated with ${I_1,...,I_n}$. 
The collection $\mathcal I$ of intervals is then called an {\em interval representation} of $G$.

As we shall see, unit interval graphs are naturally encoded by Dyck words (or Dyck paths).
We recall that a word $w$ in $\{U,D\}$ is a Dyck word if it contains as many $U$'s as $D$'s
and if all its prefixes have at least as many U's as D's.
A Dyck word is {\em irreducible} if all its proper prefixes have strictly more $U$'s than $D$'s.
Besides, the mirror of a Dyck word $w$ is the word $\overline w$ obtained by reading $w$ from right to left, changing $U$ into $D$ and $D$ into $U$. Finally, a Dyck word $w$ is called {\em palindromic} if $w=\overline w$.
Dyck words can be  represented as lattice paths, called Dyck paths,
by interpreting $U$'s as up steps $(1,1)$ and $D$'s as down steps $(1,-1)$,
and we will use both points of view interchangeably.
\medskip

Let us now explain the encoding of unit interval graphs by Dyck paths. 
Let $G$ be a unit interval graph, and $\mathcal I =(I_1,...,I_n)$
be an interval representation of $G$.
We write $I_j=[a_j,b_j]$, with $b_j=a_j+1$.
Assume without loss of generality that $a_1<\dots<a_n$
(and hence $b_1<\dots<b_n$).
Let us consider the natural order on the set $\{a_1,\dots,a_n,b_1,\dots,b_n\}$, 
\emph{i.e.}~we consider $c_1<\dots<c_{2n}$ such that 
\[ \{c_1\dots,c_{2n}\}=\{a_1,\dots,a_n,b_1,\dots,b_n\}.\]
We then define a Dyck path $w=(w_1,\dots,w_{2n})$ by
\begin{equation}\label{eq:def_w}
w_i=\begin{cases}
U & \text{if }c_i =a_k \text{ for some $k$;}\\
D & \text{if }c_i =b_k \text{ for some $k$.}
\end{cases}
\end{equation}
This construction is illustrated in \cref{Fig:Exemple_IntervalGraph_Dyck}.
Note that in this construction $a_i$, resp.~$b_i$, corresponds to the $i$-th up step,
resp.~$i$-th down step in $w$.
\begin{figure}[t]
\begin{center}
\includegraphics[width=13cm]{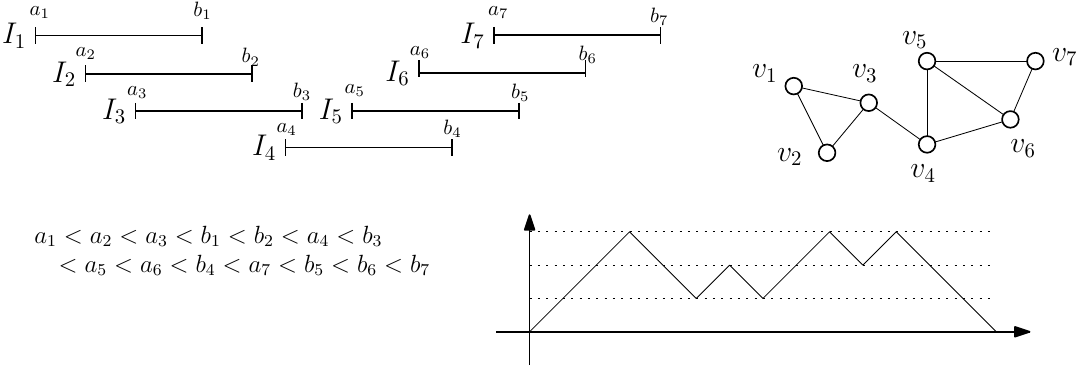}
\caption{Top: A collection $\mathcal I$ of unit intervals and the associated interval graph $G_{\mathcal I}$. Bottom: the order on the collection of starting and ending points of the intervals in $\mathcal I$ and the associated the Dyck path.}\label{Fig:Exemple_IntervalGraph_Dyck}
\end{center}
\end{figure}
\medskip

Given a Dyck path $w$, we can always find real numbers $(a_j)_{1 \le j \le n}$
 and $(b_j)_{1 \le j \le n}$ with $b_j=a_j+1$ such that \eqref{eq:def_w} holds
 (with $c_i$ the $i$-th element of the set $\{a_1,\dots,a_n,b_1,\dots,b_n\}$ in the natural order).
Moreover, all such sequences $(a_j)_{1 \le j \le n}$ and $(b_j)_{1 \le j \le n}$ 
yield the same unit interval graph, which we denote $G_w$.
However, several Dyck paths may correspond to a given unit interval graph, 
depending on the interval representation of $G$. 
In particular, it always holds that $G_w=G_{\overline w}$.

Another default of uniqueness appears when considering not connected graphs.
Let $G=G_1 \uplus G_2$ be a disjoint union of two unit interval graphs,
and let $w^{(1)}$ and $w^{(2)}$ be Dyck paths encoding $G_1$ and $G_2$.
Then both concatenations $w^{(1)} \cdot w^{(2)}$ and $w^{(2)} \cdot w^{(1)}$
are representatives of $G$.
Furthermore, it is easy to see that $G_w$ is connected if and only if $w$ is irreducible.
 
 It turns out, see \cite[Lemma 1]{saitoh2010interval}\footnote{Recall that unit interval graphs
 and proper interval graphs are the same, see \cite{bogart1999proper-unit}.},
  that mirror symmetry and disconnectedness 
are the only objections to the uniqueness of representatives.
\begin{proposition}\label{prop:uniqueness_representative_unit_interval}
If $G$ is a connected unit interval graph. Then it can be encoded
by exactly one or two (necessarily irreducible) Dyck paths. 
In the second case, the two representatives $w$
and $w'$ are mirror of each other.
\end{proposition}

\subsection{Limit of random unit interval graphs: the uniform irreducible Dyck path model}
\label{ssec:limit-uniform-Dyck-path}
In this section, we determine the limit in the Gromov--Prokhorov topology of the intersection graph of a uniform irreducible Dyck path.

We emphasize that a uniform random irreducible Dyck path
 of length $2n$ is obtained from a uniform random
Dyck path of length $2n-2$ by adding an up step at the beginning
and a down step at the end.
Therefore classical asymptotic results for uniform random Dyck paths --
such as the convergence after normalization to a Brownian excursion (recalled below) --
also hold for uniform random irreducible Dyck paths.

Let ${\bm w}$ be a uniform random irreducible Dyck path of length $2n$
 and $G_{\bm w}$ be the associated unit interval graph.
 Since $\bm w$ is irreducible, the resulting graph $G_{\bm w}$ is
 connected. However,  $G_{\bm w}$ is not uniformly distributed among connected unit interval graphs with $n$ vertices. We will address this issue in \cref{ssec:interval_removing_irred_condition}.
 
Recall that $d_{G_{\bm w}}$ denotes the graph distance in $G_{\bm w}$, and $m_{V_{G_{\bm w}}}$ the uniform measure on its vertex set, denoted $V_{G_{\bm w}}$.

 \begin{theorem}
 \label{thm:unit-interval-uniform-Dyck}
  The random mm-space $(V_{G_{\bm w}},\frac{1}{\sqrt n} d_{G_{\bm w}},m_{V_{G_{\bm w}}})$
 converges in distribution in the 
 Gromov--Prokhorov topology to $([0,1], \tfrac{1}{\sqrt 2} d_{\mathbb e},\Leb)$, where 
 $\mathbb e$ is a random Brownian excursion of length 1
 and $d_{\mathbb e}$ is defined by the formula:
 for $x<y$ in $[0,1]$, we have
 \[d_{\mathbb e}(x,y)=\int_x^y \frac{dt}{\mathbb e(t)}.\]
 \end{theorem}

To prove \cref{thm:unit-interval-uniform-Dyck}, we will need two technical lemmas.
\smallskip

The first lemma gives an asymptotic expression for the distances
in $G_{\bm w}$ in terms of the height function of $\bm w$.
Let us consider an interval representation $(I_1,\cdots, I_n)$
of $G_{\bm w}$, with $I_j=[a_j,b_j]$, such that~\eqref{eq:def_w} holds.
We assume as before that $a_1<\dots<a_n$,
and call $v_j$ the vertex of $G_{\bm w}$ corresponding to the interval $I_j$.
Also, in the following,
$h_{\bm w}(i)$ is the arrival height of the $i$-th up step in the Dyck path $\bm w$.
The proof of the following lemma is postponed to the next section.
\begin{lemma}
\label{lem:dG-sumH}
Let ${\bm w}$ be a uniform random irreducible Dyck path of length $2n$.
Then, for any $\delta$ in $(0,1/2)$,
 the following convergence holds in probability, as $n$ tends to $+\infty$:
\[ \frac1{\sqrt n} \ \sup_{\delta n \le i <j \le n-\delta n} \,\left|\, d_{G_{\bm w}} (v_i,v_j)
- \sum_{k=i}^{j-1} \frac1{h_{\bm w}(k)} \,\right| \, \longrightarrow \, 0.\]
\end{lemma}

The second lemma will allow  to estimate the sum $\sum_{k=i}^{j-1} \frac1{h_{\bm w}(k)}$ using the convergence
of $h_{\bm w}$ to the Brownian excursion.

\begin{lemma}
\label{lem:sumH-intE}
There exists a probability space with copies of $\bm w$ (one copy for each $n \ge 1$) and $\mathbb e$ such that
for any $\delta$ in $(0,1/2)$,
\[  \sup_{\delta n \le i <j \le n-\delta n} \,\left|\, \frac1{\sqrt n}\,\sum_{k=i}^{j-1} \frac1{h_{\bm w}(k)}
- \int_{i/n}^{j/n} \frac{dt}{\sqrt 2 \,\mathbb e(t)}\,\right| \, \stackrel{n\to+\infty}{\longrightarrow} \, 0\qquad \text{a.s.}.\]
\end{lemma}
\begin{proof}
We first claim that 
\begin{equation}\label{eq:conv_hauteur}
\frac{h_{\bm w}( \lfloor nx \rfloor )}{\sqrt n}
\stackrel{(d)}{\longrightarrow} \sqrt 2 \, \mathbb e(x)
\end{equation}
uniformly for all $x \in [0,1]$.
Indeed, via the classical correspondence between Dyck paths and plane trees,
for $\bm w$ a uniform Dyck path of length $2n$, the function $h_{\bm w}$ can be interpreted as the height function
of a uniform random plane tree with $n$ vertices, which corresponds to a conditioned
Galton-Watson tree with geometric offspring distribution of parameter $1/2$
(whose standard deviation is $\sigma= \sqrt 2$).
It is known, see \emph{e.g.}~\cite[Theorem 1.15]{LeGallRandomTrees},
that such a height fonction, correctly renormalized, converges to $\frac2{\sigma} \mathbb e(x)$.
The convergence holds in distribution in Skorokhod space;
however, when the limit is continuous, convergence in Skorokhod space
is equivalent to uniform convergence \cite[p. 124]{BillingsleyConv}. 
As explained above, this immediately transfers to the case where $\bm w$ is a uniform irreducible Dyck path of length $2n$. 

Using Skorokhod representation theorem, there exists a probability space with copies of $\bm w$ 
 and $\mathbb e$ on which convergence in \cref{eq:conv_hauteur} holds almost surely. We now work on this probability space.

Fix now $\delta\in (0,1/2)$,
\begin{align}\label{eq:sum_1/h-1/e}
  \sup_{\delta n \le i <j \le n-\delta n} \,\left|\, \frac1{\sqrt n}\,\sum_{k=i}^{j-1} \frac1{h_{\bm w}(k)}
- \frac1{n}  \, \sum_{k=i}^{j-1} \frac1{\sqrt 2 \,\mathbb e(k/n)} \, \right|
&\leq \sup_{\delta n \le i <j \le n-\delta n} \,\frac{1}{n}  \sum_{k=i}^{j-1} \bigg| \frac{\sqrt n}{h_{\bm w}(k)}-   \frac1{\sqrt 2 \,\mathbb e(k/n)} \bigg| \nonumber\\
&\le \frac1{n}\,\sum_{k=\delta n}^{n-\delta n-1} \bigg| \frac{\sqrt n}{h_{\bm w}(k)}- \frac1{\sqrt 2 \,\mathbb e(k/n)} \bigg|.
\end{align}
Since $\mathbb e$ is a.s.~positive
on $[\delta,1-\delta]$, \cref{eq:conv_hauteur} implies
$\displaystyle{\frac{\sqrt n}{h_{\bm w}( \lfloor nx \rfloor )} \stackrel{\text{a.s.}}{\to} \frac1{\sqrt 2 \,\mathbb e(x)}}$
uniformly for $x$ in $[\delta,1-\delta]$.
This implies that, a.s.,
\begin{equation}\label{eq:sum_1/h-1/e-to-0}  \frac1{n}\,\sum_{k=\delta n}^{n-\delta n-1} \bigg| \frac{\sqrt n}{h_{\bm w}(k)}- \frac1{\sqrt 2 \,\mathbb e(k/n)} \bigg| \, \longrightarrow \, 0.
\end{equation}
Moreover, for $\delta n \le i <j \le n-\delta n$, we have
\begin{multline*}
\left| \frac1{n}  \, \sum_{k=i}^{j-1} \frac1{\sqrt 2 \,\mathbb e(k/n)}
- \int_{i/n}^{j/n} \frac{dt}{\sqrt 2 \,\mathbb e(t)} \, \right| 
\le \int_{i/n}^{j/n} \left| \frac1{\sqrt 2 \,\mathbb e(\lfloor tn \rfloor/n)} - \frac{1}{\sqrt 2 \,\mathbb e(t)}  \right| dt \\
\le 
\sup_{x,y:\, \delta \le x,y \le 1-\delta \atop |x-y| \le 1/n} \left| \frac1{\sqrt 2 \,\mathbb e(x)} - \frac1{\sqrt 2 \,\mathbb e(y)} \right|.
\end{multline*}
Almost surely, it holds that $t \mapsto \frac1{\sqrt 2 \,\mathbb e(t)}$ is a continuous function on the interval $[\delta,1-\delta]$, and thus is uniformly continuous. The above upper bound therefore tends to $0$ as $n$ tends to $+\infty$.
Since this bound is independent from $i$ and $j$ (subject to the constraint 
 $\delta n \le i <j \le n-\delta n$) we can take the supremum over $i$ and $j$
 and conclude that
\begin{equation}\label{eq:Riemann-approx-1/e}
\sup_{\delta n \le i <j \le n-\delta n} \,\left| \frac1{n}  \, \sum_{k=i}^{j-1} \frac1{\sqrt 2 \,\mathbb e(k/n)}
- \int_{i/n}^{j/n} \frac{dt}{\sqrt 2 \,\mathbb e(t)} \, \right| \, \longrightarrow \, 0.
\end{equation}
Bringing \cref{eq:sum_1/h-1/e,eq:sum_1/h-1/e-to-0,eq:Riemann-approx-1/e} together concludes the proof of \cref{lem:sumH-intE}.
\end{proof}

\begin{proof}[Proof of  \cref{thm:unit-interval-uniform-Dyck}]
Let us write $X_n=(V_{G_{\bm w}},\frac{1}{\sqrt n} d_{G_{\bm w}},m_{V_{G_{\bm w}}})$
and $X_\infty=([0,1],\tfrac{1}{\sqrt 2} d_{\mathbb e},\Leb)$. 
For proving the Gromov--Prokhorov convergence of $X_n$ to $X_\infty$ the strategy is to use \cref{th:Loehr}. For this purpose we introduce on $X_n \times X_\infty$ a relation  $R_{n,\delta}$ and a distribution $\nu$ which allow to bound the box distance $\Box(X_n,X)$ (see \cref{def:BoxDistance}).

Fix $\delta>0$.
Let $R_{n,\delta} \subseteq X_n \times X_\infty$ be the relation given by $R_{n,\delta} := \{ (v_{1+\lfloor xn \rfloor},x), \, \delta \le x \le 1-\delta\}$,  where $v_k$ denotes, as before, the vertex of $G_{\bm w}$ corresponding to the $k$-th interval of an interval representation of $G_{\bm w}$.
Let also $\nu$ be the distribution of $(v_{1+\lfloor nU \rfloor},U)$
where $U$ is uniform in $[0,1]$. Since $1+\lfloor nU \rfloor$ is uniform in $[n]$ the first marginal of $\nu$ is $m_{V_{G_{\bm w}}}$, so that $\nu$ is a coupling between $m_{V_{G_{\bm w}}}$ and $\Leb$.
By construction we have $\nu(R_{n,\delta}) = 1-2\delta$.

The discrepancy (see again \cref{def:BoxDistance}) of $R_{n,\delta}$ is equal to
\[ \disc(R_{n,\delta}) = \sup_{(x_1,x'_1), (x_2,x'_2) \in R_{n,\delta}} 
\left| \frac{1}{\sqrt n}\, d_{G_{\bm w}}(x_1,x_2) - \frac{1}{\sqrt 2}\, d_{\mathbb e}(x'_1,x'_2) \right|.\]
We have 
\begin{align*}
\disc(R_{n,\delta}) &= \sup_{\delta \le  x<y  \le 1-\delta} \,\left|\, 
\frac{1}{\sqrt n}\, d_{G_{\bm w}}(v_{1+\lfloor xn \rfloor},v_{1+\lfloor yn \rfloor}) 
- \int_{x}^{y} \frac{dt}{\sqrt 2 \,\mathbb e(t)}\,\right| \\
&\le  \sup_{1+\lfloor n\delta \rfloor \le i<j  \le 1+\lfloor n(1-\delta)\rfloor} \,\left|\, 
\frac{1}{\sqrt n}\, d_{G_{\bm w}}(v_i,v_j) 
- \int_{i/n}^{j/n} \frac{dt}{\sqrt 2 \,\mathbb e(t)}\,\right| + \frac{2}{n} \sup_{t \in [\delta,1-\delta]} \frac{1}{\sqrt 2 \,\mathbb e(t)}.
\end{align*}
From \cref{lem:dG-sumH,lem:sumH-intE},
we know that the first summand of this upper bound tends to $0$ in probability as $n$ tends to $+\infty$.
The second one tends to $0$ as well, since $ \sup_{t \in [\delta,1-\delta]} \frac{1}{\sqrt 2 \,\mathbb e(t)}$
is a random variable independent of $n$. We conclude that $\disc(R_{n,\delta})$ tends in probability to $0$.

By definition of the box distance, 
$$
\Box(X_n,X_\infty) \le \max(\disc(R_{n,\delta}),2\delta),
$$
so that $\mathbb P( \Box(X_n,X_\infty) > 2\delta) \le \mathbb P( \disc(R_{n,\delta}) > 2\delta)$.
 The latter tends to $0$ since $\disc(R_{n,\delta})$ tends in $0$ in probability, and therefore,
 $\mathbb P(\Box(X_n,X_\infty) > 2\delta)$ tends to $0$.
 This holds for any $\delta>0$,
i.e.~$X_n$ tends to $X_\infty$ in probability
for the box distance, in the probability space constructed in
\cref{lem:sumH-intE}.
We conclude that in the original probability space, $X_n$ tends to $X_\infty$ in distribution for the Gromov--Prokhorov topology,
as wanted.
\end{proof}

\subsection{Proof of \cref{lem:dG-sumH} }
\label{ssec:preuve-lemme-distance}

Fix an irreducible Dyck path $w$ of length $2n$. 
We start by explaining how to compute distances in $G_w$.
We consider as usual an interval representation of $G_w$, denoted $\mathcal I=(I_1,...,I_n)$, with $I_j = [a_j,b_j]$ for all $i$, and $a_1 < a_2 < \dots < a_n$. 
We also denote $v_j$ the vertex of $G_w$ represented by $I_j$. 
For $i \le n$, we let
$f_w(i)$ be the number of up steps between the $i$-th up step (excluded) and the $i$-th down step in $w$.
Note that $f_w(i)>0$ for all $i<n$ since $w$ is irreducible.
Recall that in the correspondance between the Dyck path $w$ and the unit interval graph $G_w$,
the $i$-th up step and the $i$-th down step in $w$ corresponds to the bound $a_i$ and $b_i$ of the interval $I_i$. Hence, by definition, $f_w(i)$ is the maximal $k$ such that
the interval $I_{i+k}$ starts before the end of $I_i$, \emph{i.e.}~it is the maximal $k$ such that $v_i$ and $v_{i+k}$ are connected in $G_{w}$.
This property allows to compute distances in $G_w$ using the function $f_w$ (see \cref{Fig:Notation_LemmaDistances}).
\begin{lemma}
\label{lem:distance_unit}
Let $w$ be an irreducible Dyck path of length $2n$ and take $i<j$ in $[n]$.
Define $i_0=i$ and recursively $i_{m+1}=i_m+f_w(i_m)$ until $i_m \ge j$.
One has
\begin{equation}\label{eq:dGw}
 d_{G_w}(v_i,v_j)
= \left\lceil \sum_{k=i}^{j-1} \frac{1}{f_w(\max\{i_m: i_m \le k\})}  \right\rceil,
\end{equation}
where $\lceil x \rceil$ is the smallest integer greater than or equal to $x$.
\end{lemma}
\begin{figure}
\begin{center}
\includegraphics[width=13cm]{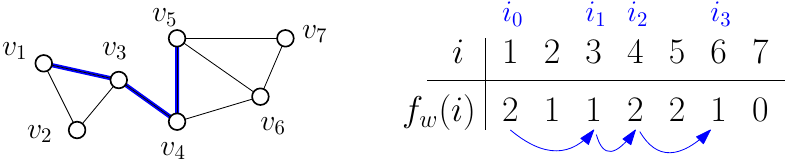}
\caption{Illustration of \cref{lem:distance_unit} and its notation. Left: the unit interval graph associated to the Dyck path $w$ of \cref{Fig:Exemple_IntervalGraph_Dyck}, and a shortest path from $v_1$ to $v_5$. Right: the corresponding function $f_w$, and how to read on it that $ d_{G_w}(v_1,v_5)=\lceil 1/f_w(i_0)+  1/f_w(i_0)+ 1/f_w(i_1)+ 1/f_w(i_2) \rceil = \lceil 1/2+ 1/2+ 1/1+ 1/2 \rceil =3$. }\label{Fig:Notation_LemmaDistances}
\end{center}
\end{figure}

\begin{proof}
It is clear from the interval representation of $G_w$ that for any $i \le k \le j$ we have
\[ d_{G_w}(v_i,v_j) \ge d_{G_w}(v_k,v_j).\]

Hence finding a shortest path from $v_i$ to $v_j$ ($i<j$) can be realized by the following gready procedure:
\begin{itemize}
\item if $v_i$ is connected to $v_j$, we have a path of length $1$;
\item otherwise, we find the neighbor of $v_i$ with greatest label,
which is $v_{i+f_w(i)}$ as explained above. We take the edge $\{v_i,v_{i+f_w(i)}\}$,
concatenated with a shortest path from $v_{i+f_w(i)}$ to $v_j$ built recursively by the same procedure.
\end{itemize}
In terms of distance, this yields (for $i<j$) 
\[ d_{G_w}(v_i,v_j) = \begin{cases}
1 &\text{ if $j \le i +f_w(i)$;}\\
1+d_{G_w}(v_{i+f_w(i)},v_j)&\text{ otherwise.}
\end{cases}\]
It is easy to verify that the right-hand side of \eqref{eq:dGw} satisfies the same recursive characterization, proving the lemma. (The integer part guarantees that the formula is true even if $i_m > j$.)
\end{proof}

Let us now consider a uniform random irreducible Dyck path $\bm w$ of length $2n$.
Our goal is to show that $f_{\bm w}(\max\{i_m: i_m \le k\})$ is close to $h_{\bm w}(k)$.
We first show that $f_{\bm w}(i)$ and $h_{\bm w}(i)$ are typically close to each other.
We start with a classical concentration type results for Dyck paths.
\begin{lemma}\label{lem:balanced_intervals}
Let $\bm w$ be a uniform random irreducible Dyck path of length $2n$.
Fix $\eps_n=n^{-0.1}$. Then, with probability tending to $1$,
for all intervals $I \subset  [2n]$
 of size at least $n^{0.4}$, the proportion of up steps in $I$
 lies in $[\frac12 - \eps_n; \frac12 + \eps_n]$.
\end{lemma}
\begin{proof}

Let $\tilde{\bm w}$ be a uniform random binary word of length $2n$. Starting from the standard estimate $\mathbb{P}(\tilde{\bm w}\text{ is an irreducible Dyck path})= \frac1n \binom{2n-2}{n-1} \times 2^{-2n}
=\Theta(n^{-3/2})$, we get that, for any event $E_n$,
$$
\mathbb{P}(\bm w\in E_n)=\mathcal{O}\left(n^{3/2}\mathbb{P}(\tilde{\bm w}\in E_n)\right).
$$
Let $E_n$ be the following event:
there exists an interval $I \subset [2n]$ of size at least $n^{0.4}$
such that the proportion of up steps in $I$
 is not in $[\tfrac12 - \eps_n; \tfrac12 + \eps_n]$.
By the union bound
\begin{align*}
 \mathbb{P}\left(\tilde{\bm w}\in E_n\right)&\leq 
 (2n)^2  \max_{I \subset [2n]\,\text{s.t.}\,|I|\geq n^{0.4} }\mathbb{P}\left(\text{the proportion of 'up' in $I$ is not in $[\tfrac12 - \eps_n; \tfrac12 + \eps_n]$}\right)\\
 &\leq (2n)^2 \max_{s\geq n^{0.4} }\mathbb{P}\left(\left|\mathrm{Binom}(s,1/2)-s/2 \right| > s\eps_n\right)\\
 &\leq (2n)^2 \max_{s\geq n^{0.4} }2\exp\left(-2 s\eps_n^2\right),\text{ using the Hoeffding inequality}\\
 &\leq 8 n^2 \exp\left(-2 n^{0.4-0.2}\right).
\end{align*}
This proves that $\mathbb{P}(\bm w\in E_n)\leq e^{-cn^\alpha}$ for some $c,\alpha>0$, hence concluding the proof.
\end{proof}
\begin{corollary}
\label{corol:HSurF}
Let $\bm w$ be a uniform random irreducible Dyck path of length $2n$.
Fix $\eps_n=n^{-0.1}$.
With probability tending to $1$, the following holds.
For all $i$ such that either $h_{\bm w}(i) \ge n^{0.4}$ or $f_{\bm w}(i) \ge n^{0.4}$,
the quotient $\frac{h_{\bm w}(i)}{f_{\bm w}(i)}$ belongs to $[1-5\eps_n;1+5\eps_n]$.
\end{corollary}
\begin{proof}
We observe that, by definition of $h_{\bm w}(i)$,
 there are $i-h_{\bm w}(i)$ down steps before the $i$-th up step of $w$.
 Hence there are $h_{\bm w}(i)$ down steps between the $i$-th up step (excluded) and the $i$-th down step of ${\bm w}$ (included). By definition, the number of up steps in the same interval is $f_{\bm w}(i)$.
 Hence if either $h_{\bm w}(i) \ge n^{0.4}$ or $f_{\bm w}(i) \ge n^{0.4}$,
 this interval has length at least $n^{0.4}$ and \cref{lem:balanced_intervals} applies.
 We get that $\frac{f_{\bm w}(i)}{f_{\bm w}(i) + h_{\bm w}(i)}$ belongs to $[\frac12 - \eps_n; \frac12 + \eps_n]$. 
 Elementary manipulations then imply that, for $n$ large enough, $\frac{ h_{\bm w}(i)}{f_{\bm w}(i)}$ belongs to $[1 - 5\eps_n; 1 + 5\eps_n]$, 
 concluding the proof of the lemma.
\end{proof}

\begin{corollary}
\label{corol:dG-SumH}
Let $\bm w$ be a uniform random Dyck path of length $2n$. For any $\delta$ in $(0,1/2)$,
 we have the following convergence in probability:
\[ \frac1{\sqrt n} \ \sup_{\delta n \le i <j \le n-\delta n} \,\left|\, d_{G_{\bm w}} (v_i,v_j)
- \sum_{k=i}^{j-1} \frac1{h_{\bm w}(\max\{i_m: i_m \le k\})} \,\right| \, \longrightarrow \, 0.\]
\end{corollary}
\begin{proof}
It is enough to check that, for any $\delta$ in $(0,1/2)$,
 we have the following convergence in probability:
\begin{equation}
\label{eq:estimate-1surH-1surF} \sup_{\delta n \le \ell \le n -\delta n} \sqrt n \left| \frac1{f_{\bm w}(\ell)}- \frac1{h_{\bm w}(\ell)}\right| \, \longrightarrow \, 0.
\end{equation}
Indeed, if this holds, it suffices to use \cref{eq:dGw}
and to sum the above estimate for $\ell=\ell(k)=\max\{i_m: i_m \le k\}$
for $k$ in $\{i,\cdots,j-1\}$.

The left-hand side of \eqref{eq:estimate-1surH-1surF} rewrites as
\[ \sup_{\delta n \le \ell \le n -\delta n} \left( \frac{\sqrt n}{h_{\bm w}(\ell)} \left| \frac{h_{\bm w}(\ell)}{f_{\bm w}(\ell)} -1 \right| \right)
\le  \frac{\sqrt n}{\inf_{\delta n \le \ell \le n -\delta n} h_{\bm w}(\ell)}
\cdot \sup_{\delta n \le \ell \le n -\delta n} \left| \frac{h_{\bm w}(\ell)}{f_{\bm w}(\ell)} -1 \right|.\]
Recall from \eqref{eq:conv_hauteur} that $\frac{1}{\sqrt n} h_{\bm w}(\lfloor nx \rfloor)$ converges in distribution to $\sqrt{2}\,\mathbb e(x)$.
We then have
\[\frac{\sqrt n}{\inf_{\delta n \le \ell \le n -\delta n} h_{\bm w}(\ell)}
\to \frac{1}{\inf_{x \in [\delta,1-\delta]} \sqrt 2\,\mathbb e(x)},\]
in distribution, as $n$ tends to $+\infty$. Note that the right-hand-side
is a.s.~finite since the Brownian excursion does not vanish in $[\delta,1-\delta]$.
Moreover, with probability tending to $1$, we have that $h_{\bm w}(\ell) \ge n^{0.4}$
for all $\ell$ in  $[\delta n,n-\delta n]$. Thus we can apply \cref{corol:HSurF},
and we get, that with probability tending to $1$
\[\sup_{\delta n \le \ell \le n -\delta n} \left| \frac{h_{\bm w}(\ell)}{f_{\bm w}(\ell)} -1 \right| \le 5 \eps_n =5 n^{-0.1}.\]
Bringing everything together proves \cref{eq:estimate-1surH-1surF},
and thus \cref{corol:dG-SumH}.
\end{proof}

\begin{lemma}\label{lem:diff-hauteur}
The following holds with probability tending to $1$.
For any $i\le k<j$, we have
\[|h_{\bm w}(k) - h_{\bm w}(\max\{i_m: i_m \le k\})| \le n^{0.45}.\]
\end{lemma}
\begin{proof}
Set $k':=\max\{i_m: i_m \le k\}$ and let $I$ be the interval between the $k'$-th up step (excluded) and the $k$-th up step (included).

We first bound the length of this interval.
Since $h_{\bm w}$ converges after normalization in space
 by $\sqrt n$ to the Brownian excursion, with probability tending to 1, it holds that 
\[ \sup_{i\le n} h_{\bm w}(i) \le n^{0.52}. \]
Using \cref{corol:HSurF}, we get that, with probability tending to 1, 
\[ \sup_{i\le n} f_{\bm w}(i) \le n^{0.53}. \]
By construction of the sequence $(i_m)$, the number of up-steps in $I$ is bounded as follows:
\[ |k-k'| \le f_{\bm w}(k') \le 
 \sup_{i\le n} f_{\bm w}(i) \le n^{0.53},\]
 where the last inequality holds with probability tending to $1$.
 By \cref{lem:balanced_intervals}, with probability tending to 1,
 the number of down steps satisfy a similar inequality up to a factor tending to $1$. We conclude that the inequality $|I| \le n^{0.54}$ holds with probability tending to $1$.
 
 We now observe that $h_{\bm w}(k) - h_{\bm w}(k')$
    is the difference between the number of up and down steps in 
    the interval $I$, and we distinguish two cases.
 \begin{itemize}
 \item If $|I| \le n^{0.45}$, then trivially
 \[ |h_{\bm w}(k) - h_{\bm w}(k')| 
    \le | I | \le  n^{0.45}. \]
    \item Otherwise, $n^{0.45} \le |I| \le n^{0.54}$.
    By \cref{lem:balanced_intervals}, with probability tending to 1, we have
    \[ |h_{\bm w}(k) - h_{\bm w}(k')| 
    \le 2 \eps_n |I| \le 2n^{-0.1} n^{0.54} \le  
    n^{0.45}. \qedhere\]
 \end{itemize}
\end{proof}

\begin{proof}[Proof of \cref{lem:dG-sumH}]
Fix $\delta>0$. Using again the convergence of $h_{\bm w}$
to the Brownian excursion,
with probability tending to 1, we have
\[\inf_{\ell \in [\delta n,n-\delta n]} h_{\bm w}(\ell) \ge n^{0.49}.\]
Thus (using also \cref{lem:diff-hauteur}), with probability tending to 1,
for any $i\le k<j$ in $[\delta n,n-\delta n]$,
\[ \left| \frac1{h_{\bm w}(k)} - \frac1{h_{\bm w}(\max\{i_m: i_m \le k\})}\right|
 \le \frac{ |h_{\bm w}(k) - h_{\bm w}(\max\{i_m: i_m \le k\})|}{h_{\bm w}(k) \, h_{\bm w}(\max\{i_m: i_m \le k\})} \le \frac{n^{0.45}}{n^{0.98}} =n^{-0.53}.\]
 Summing over $k$, we get that, with probability tending to $1$,
 \[ \frac{1}{\sqrt n} \sup_{\delta n \le i <j \le n-\delta n} \sum_{k=i}^{j-1} \left| \frac1{h_{\bm w}(k)} - \frac1{h_{\bm w}(\max\{i_m: i_m \le k\})}\right| \le n^{-0.03}.\]
 Together with \cref{corol:dG-SumH}, this yields
 \[ \frac1{\sqrt n} \ \sup_{\delta n \le i <j \le n-\delta n} \,\left|\, d_{G_{\bm w}} (v_i,v_j)
- \sum_{k=i}^{j-1} \frac1{h_{\bm w}(k)} \,\right| \, \longrightarrow \, 0,\]
 as wanted.
\end{proof}

\subsection{From uniform irreducible Dyck paths to uniform unit interval graphs}
\label{ssec:interval_removing_irred_condition}
Let $\bm G_n$ be a uniform (possibly disconnected)
unit interval graph with $n$ vertices.
The goal of this section is to prove that $\bm G_n$ has the same
Gromov-Prokhorov limit as that found for $G_{\bm w}$ in \cref{thm:unit-interval-uniform-Dyck}.
As a first step, we prove the result 
for a uniform {\em connected} unit interval graph ${\bm C_n}$  with $n$ vertices.
In the sequel, we use $\dTV(\mu,\nu)$ 
for the total variation distance between probability measures $\mu$ and $\nu$,
and by extension, for random variables $X$ and $Y$,
we write $\dTV(X,Y)$ for the total variation distance between their laws.

\begin{lemma}
\label{lem:dTV-Gw-Cn}
Let $\bm w$ be a uniform irreducible Dyck path of length $2n$ and $\bm C_n$
be a uniform {\em connected} unit interval graph with $n$ vertices. 
It holds that 
\[\lim_{n \to +\infty} \dTV(G_{\bm w},\bm C_n) =0.\]
\end{lemma}
\begin{proof}
Let us consider the map $w \mapsto G_w$ mapping {\em irreducible} Dyck paths
of length $2n$
to connected unit interval graphs with $n$ vertices.
From \cref{prop:uniqueness_representative_unit_interval},
each connected unit interval graph has either 1 or 2 pre-images.
The lemma will follow if we show that the probability that $\bm C_n$
has exactly 1 pre-image tends to $0$ as $n$ tends to $+\infty$.

But a connected unit interval graph has exactly 1 pre-image $w$ if and only if $w$ is irreducible and palindromic. 
Moreover
\[\frac{ \#\left\{
\begin{tabular}{c}palindromic irreducible \\Dyck paths of length $2n$ \end{tabular}
\right\} }{ \#\left\{
\begin{tabular}{c} irreducible Dyck\\ paths of length $2n$ \end{tabular}
\right\} } = 
\frac{ \#\left\{
\begin{tabular}{c} Dyck prefixes \\ of length $n-1$ \end{tabular}
\right\} }{ \#\left\{
\begin{tabular}{c}  Dyck paths \\ of length $2n-2$ \end{tabular}
\right\} }
=\frac{\displaystyle \binom{n-1}{\lfloor (n-1)/2 \rfloor}}{\displaystyle \frac1n \binom{2n-2}{n-1}},\]
where the enumeration for Dyck prefixes can be found, 
\emph{e.g.}, in \cite[Ex.$6.19$, p.$219$]{Sta99}.
The right-hand side obviously tends to $0$, ending the proof of the lemma.
\end{proof}

\begin{corollary}
\label{corol:convergenceC_n}
\cref{thm:unit-interval-uniform-Dyck} holds 
true with $\bm C_n$ instead of $G_{\bm w}$.\medskip
\end{corollary}

We now consider a uniform (possibly disconnected)
unit interval graph $\bm G_n$ with $n$ vertices.
We use the standard notation $X_n=\mathcal{O}_{\mathbb{P}}(1)$ to say that a sequence
of random variables $X_n$ is stochastically bounded\footnote{\ \emph{i.e.} for every $\eps>0$ there exist constants $k_\eps,n_\eps$ such that for $n\geq n_\eps$ one has $\mathbb{P}(|X_n|\leq k_\eps)\geq 1-\eps$.}.
\begin{lemma}
\label{lem:Gn-large-components}
Let $\bm G_n$ be as above and let $L_n$ be the size of its largest connected component.
Then $n-L_n=\mathcal{O}_{\mathbb{P}}(1)$.
\end{lemma}
\begin{proof}
Let $C(z)$ and $G(z)$ be the ordinary generating series of connected and general unit interval graphs with respect to the number of vertices.
Since a general graph is a multiset of connected graphs, using \cite[Theorem 1.1]{Violet},
we have
\[G(z) = \exp\left( \sum_{k \ge 1} \frac{1}{k} C(z^k) \right).\]
We write this as $G(z)=F(C(z),z)$ where
\[F(w,z) = \exp(w) \cdot \exp\left( \sum_{k \ge 2} \frac{1}{k} C(z^k)\right).\]
From \cref{prop:uniqueness_representative_unit_interval}, we get that
\[C(z) = \tfrac{1}{2} I(z) +\tfrac{1}{2} P(z),\]
where $I(z)$ and $P(z)$ are respectively the series of irreducible Dyck paths 
and of palindromic irreducible Dyck paths.
Since irreducible Dyck paths of length $2n+2$ are in one-to-one correspondence with Dyck paths of length $2n$, 
the series $I(z)$ has radius of convergence $\rho=1/4$ and a square-root singularity.
Moreover, 
\[P(z)=1+ \sum_{n \ge 1} \binom{n-1}{\lfloor (n-1)/2 \rfloor} z^n\]
  has radius of convergence $1/2$. 
Therefore $C(z)$ has a square-root singularity at $\rho=1/4$. 
In particular, it is of {\em algebraic-logarithmic type}, 
as defined in \cite[Definition 1]{Gourdon}.
Therefore we can apply \cite[Theorem 1]{Gourdon} -- in this reference,
the author only considers the case where the function $F$ depends on one variable $w$,
but his proof readily extends to the case of a bivariate function $F(w,z)$, provided that
it is analytic at $(w,z)=(C(\rho),\rho)$.
We get that $n-L_n$ converges to a discrete law, proving that it is stochastically bounded.
\end{proof}

We can now prove the main result of this section, whose statement we recall:

\medskip

\noindent{\bf \cref{thm:unit-interval}}
\emph{ Let $\bm G_n$ be a uniform random unit interval graph with $n$ vertices.
  The random mm-space $(\bm G_n,\frac{1}{\sqrt n} d_{\bm G_n},m_{V_{\bm G_n}})$
 converges in distribution in the 
 Gromov--Prokhorov topology to $\linebreak ([0,1], \tfrac{1}{\sqrt 2} d_{\mathbb e},\Leb)$.
}
 \begin{proof}
 We let $\bm G^0_n$ be the largest connected component of $\bm G_n$.
 From \cref{lem:Gn-large-components}, $\bm G^0_n$ has size $L_n=n-\mathcal{O}_{\mathbb{P}}(1)$. 
 Moreover, conditioned to $L_n$, $\bm G^0_n$ is a uniform connected unit interval graph
 with $L_n$ vertices.
  From \cref{corol:convergenceC_n},
  \cref{thm:unit-interval-uniform-Dyck} holds true with 
  a uniform connected unit interval graph $\bm C_n$ instead of $G_{\bm w}$.
 Therefore the random 
 mm-space $(\bm G^0_n,\frac{1}{\sqrt L_n} d_{\bm G^0_n},m_{V_{\bm G^0_n}})$
 converges in the GP topology to $([0,1], \tfrac{1}{\sqrt 2} d_{\mathbb e},\Leb)$.
 The theorem follows because $\dTV(m_{V_{\bm G^0_n}},m_{{V_{\bm G_n}}}) = \frac{n-L_n}n$
 and $d_{\bm G^0_n}$ is the distance $d_{\bm G_n}$ restricted to $\bm G^0_n$.
 \end{proof}

\subsection{Number of copies of $K_k$}
\label{ssec:nb-clique-unit}
In this section, we find the asymptotic behavior 
of the number of copies of the complete graph $K_k$ in 
a uniform random unit interval graph $\bm G_n$.
Unlike the case of permutation and circle graphs,
this does not follow directly from our scaling limit result,
but builds on the same intermediate considerations.

As above, we first consider the unit interval graph $G_{\bm w}$ associated
with a uniform irreducible Dyck path $\bm w$ of length $2n$.
We start with the following deterministic lemma,
where we use the notation $f_w(i)$ from \cref{ssec:preuve-lemme-distance}.

\begin{lemma}
\label{lem:count-cliques-Gw}
Let $w$ be an irreducible Dyck path, and $G_w$ its associated unit interval graph.
The following holds:
\[\#\set{\text{cliques of size $k$ in }G_w }  = \sum_{i=1}^n \binom{f_w(i)}{k-1}.\]
\end{lemma}
\begin{proof}
Consider a $k$-tuple of vertices $(v_{i_1},v_{i_2},\dots,v_{i_k})$ of $G_w$,
where the vertices of $G_w$ are labeled as in \cref{ssec:limit-uniform-Dyck-path}
and where $i_1<i_2<\dots<i_k$.
Recall that the vertex $v_{i_j}$ corresponds to some interval $[a_{i_j},b_{i_j}]$.

We claim that this $k$-tuple induces a clique in $G_w$
if and only if $i_k \leq i_1+f_w(i_1)$.
This condition is clearly necessary, since this is a necessary condition for $i_k$
to be connected to $i_1$. Conversely, if $i_k \leq i_1+f_w(i_1)$,
then all of $a_{i_2}$, \dots, $a_{i_k}$ belong to $[a_{i_1},b_{i_1}]$.
Since all intervals have unit length, all intervals  $[a_{i_j},b_{i_j}]$ intersect each other,
and the vertices $(v_{i_1},v_{i_2},\dots,v_{i_k})$ indeed induce a clique $K_k$ in $G_w$.

We now count such $k$-tuples, grouping them by the value of $i_1$.
For $i_1$ in $[n]$, there are $\binom{f_w(i_1)}{k-1}$ ways to choose $i_2 < \dots <i_k$
larger than $i_1$ such that $i_k \leq i_1+f_w(i_1)$. The formula in the lemma follows immediately.
\end{proof}

This lemma allows to find the asymptotic behavior of the number of cliques of size $k$
in $G_{\bm w}$, where $\bm w$ is a uniform random irreducible Dyck path of length $2n$.
\begin{lemma}
\label{lem:asympt-cliques-Gw}
For any $K\ge 1$, we have the following joint convergence in distribution:
\[ \left(\frac{\#\set{\text{cliques of size $k$ in }G_{\bm w} }}{n^{\frac{k+1}2}} \right)_{2\leq k \le K}
\to \left( \frac{2^{(k-1)/2}}{(k-1)!}   \int_0^1 \mathbb e(t)^{k-1} dt \right)_{2\leq k \le K},\]
where $\mathbb e(t)$ is a Brownian excursion.
\end{lemma}
\begin{proof}
For $k \le K$, we let $N_k(G_{\bm w})$ be the number of cliques of size $k$ in $G_{\bm w}$. 
From \cref{lem:count-cliques-Gw}, we have
\begin{multline*}
n^{-\frac{k+1}{2}} N_k(G_{\bm w}) = n^{-\frac{k+1}{2}} \sum_{i=1}^n \binom{f_{\bm w}(i)}{k-1} =n^{-\frac{k+1}{2}} \sum_{i=1}^n \left( \frac{f_{\bm w}(i)^{k-1}}{(k-1)!} + O\big(f_{\bm w}(i)^{k-2}\big) \right)\\
=\frac{1}{n}\left(\sum_{i=1}^n \frac{1}{(k-1)!} \left(\frac{f_{\bm w}(i)}{\sqrt n}\right)^{k-1} \right) + \frac{1}{\sqrt n}\ O\left(\sup_{i} \left(\frac{f_{\bm w}(i)}{\sqrt n}\right)^{k-2} \right).
\end{multline*}
Using the convergence of $\frac{h_{\bm w}}{\sqrt n}$ to a Brownian excursion and \cref{corol:HSurF}, we see that $\frac{f_{\bm w}(i)}{\sqrt n}$ is bouded almost surely. Consequently, for $k \ge 2$,
\[n^{-\frac{k+1}{2}} N_k(G_{\bm w}) = \frac{1}{n}\left(\sum_{i=1}^n \frac{1}{(k-1)!} \left(\frac{f_{\bm w}(i)}{\sqrt n}\right)^{k-1} \right) +o_P(1),\]
where, as usual, the notation $o_P(1)$ represents a random variable
converging to $0$ in probability.
Using \cref{corol:HSurF} and observing that terms with $f_{\bm w}(i)<n^{0.4}$ and $h_{\bm w}(i)<n^{0.4}$
have a negligible contribution, we get that with a probability tending to one
\begin{multline*}
n^{-\frac{k+1}{2}} N_k(G_{\bm w}) = \frac{1}{n}\sum_{i=1}^n \frac{1}{(k-1)!} \left(\frac{h_{\bm w}(i)}{\sqrt n}\right)^{k-1} +o_P(1) \\
= \frac{1}{(k-1)!} \int_0^1 \left(\frac{h_{\bm w}(\lfloor nt \rfloor)}{\sqrt n}\right)^{k-1} dt +o_P(1).
\end{multline*}
Recall from \cref{eq:conv_hauteur} that 
 the random function $\frac{h_{\bm w}(\lfloor nt \rfloor)}{\sqrt n}$ converges uniformly
in distribution to $\sqrt 2 \,\mathbb e(t)$, in the space of continuous functions on $[0,1]$ equipped with uniform convergence.
 Since  $f\mapsto \left(1,\int f,\int f^2,\dots, \int f^{K-1}  \right)$ is continuous on that space this implies 
that
\[n^{-\frac{k+1}{2}} N_k(G_{\bm w}) \rightarrow \frac{2^{(k-1)/2}}{(k-1)!} \int_0^1 \mathbb e(t)^{k-1} dt,\]
jointly for $2\leq k\leq K$. The lemma is proved.
\end{proof}
From \cref{lem:dTV-Gw-Cn}, a version of \cref{lem:asympt-cliques-Gw} where $G_{\bm w}$
is replaced by a uniform random connected unit interval graph $\bm C_n$
also holds. Finally, mimicking the proof of \cref{thm:unit-interval},
the result also holds for a uniform random (nonnecessarily connected) unit interval
graph $\bm G_n$, concluding the proof of \cref{thm:number-of-cliques}.

%


\appendix
\section{Proofs of two technical results}\label{Sec:Appendice}

{\small 

\subsection{Proof of \cref{prop:prime=indec}: indecomposable matchings and split-prime circle graphs}\label{Sec:AppendiceMatching}

Let us recall the statement of the Proposition.

\medskip

\noindent{\bf Proposition \ref{prop:prime=indec}.}
\emph{
Let $G$ be a circle graph and $\match$ be a matching that represents $G$.
Then $G$ is split-prime if and only if $\match$ is indecomposable.
}

\begin{proof}
Let $n$ denote the number of vertices of $G$. If $n\leq 3$ then $G$ is trivially split-prime and $\match$ cannot be decomposable so there is nothing to prove. Assume $n\geq 4$.
\smallskip

\noindent{\bf Proof of $\match$ decomposable $\Rightarrow$ $G$ has a nontrivial split.}
Let $C_1, C_2, C_3, C_4$ be the partition of $[2n]$ associated to the $k$-decomposition of $\match$. 
Let $V_{odd}$ (resp. $V_{even}$) be the set of vertices of $G$ corresponding to chords of $C_1 \cup C_3$ (resp. $C_2\cup C_4$). The sets $V_{odd}$ and $V_{even}$ both contain at least two vertices and form a nontrivial split, whose cut vertex set $V^{cut}_{odd}$ (resp.~$V^{cut}_{even}$)
consists of chords between $C_1$ and $C_3$ (resp.~of chords between $C_2$ and $C_4$). See \cref{fig:k-dec-matching}~(right).
\medskip

\noindent{\bf Proof of $G$ has a nontrivial split $\Rightarrow$ $\match$ decomposable. }
We distinguish three cases.
\smallskip

{\em Case 1: $\match$ has a chord $\{a,a{+}1\}$ for some $a$} (where $a{+}1$ is interpreted mod $2n$,
as well as $a{+}2$ below).
Let $b$ be such that $\{a{+}2,b\} \in \match$.
Then $\match$ admits a decomposition as in \cref{def:decomposable_matching},
where one of the sets $C_i$ is $\{a,a{+}1,a{+}2\}$ and another is $\{b\}$.
Hence the matching $\match$ is indeed decomposable.
\smallskip

{\em Case 2: $G$ is disconnected but $\match$ has no chord of the form $\{a,a{+}1\}$.}
Let $V'$ be the set of vertices of a connected component of $G$.
Each vertex in $V'$ corresponds to a pair $\{a,b\}$ in $\match$,
and we denote by $I$ the union of such pairs. Up to choosing another connected
component $V'$, we may assume that $1 \notin I$. Let $C_2$ be the integer interval
$[\min(I),\max(I)]$ (in particular, all chords of $V'$ have both extremities in $C_2$)
and $C_1$ its complement in $[2n]$.
We claim that there is no chord from $C_1$ to $C_2$ in $\match$. 
Indeed such a chord would necessarily cross
a chord corresponding to a vertex in $V'$ (since such chords form a connected set in the unit disk containing $\min(I)$ and $\max(I)$).
Thus this chord would itself correspond to a vertex in $V'$,
which is impossible since it has an extremity in $C_1$.

We have proved that $\match$ has only chords from $C_1$ to $C_1$,
and from $C_2$ to $C_2$.
Since $\match$ has no chord of the form $\{a,a{+}1\}$,
each of $C_1$ and $C_2$ has size at least 4,
and $(C_1,C_2,\emptyset,\emptyset)$ is a decomposition as in \cref{def:decomposable_matching}. Thus $\match$ is decomposable.
\smallskip

{\em Case 3: $G$ is connected.}
Since $G$ is assumed to be split-decomposable,
its vertex set admits a nontrivial split $\{V_{odd},V_{even}\}$,
with corresponding cut vertex sets
 $V_{odd}^{cut}$ and $V_{even}^{cut}$.
 Vertices in $G$ correspond to chords in the matching $\match$,
 and we let $S_{odd}$ (resp.~$S_{even}$, $S_{odd}^{cut}$ and $S_{even}^{cut}$)
  be the set of points belonging to a chord in 
  $V_{odd}$ (resp.~$V_{even}$, $V_{odd}^{cut}$ and $V_{even}^{cut}$).

\begin{figure}
\[ \includegraphics{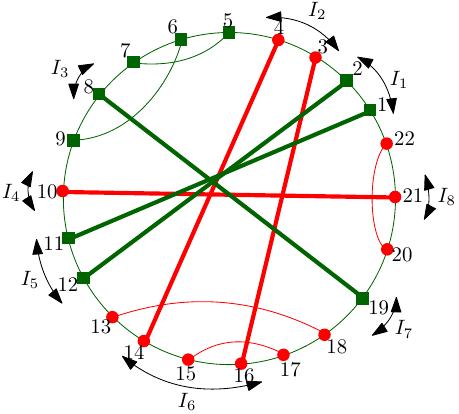} \qquad \includegraphics{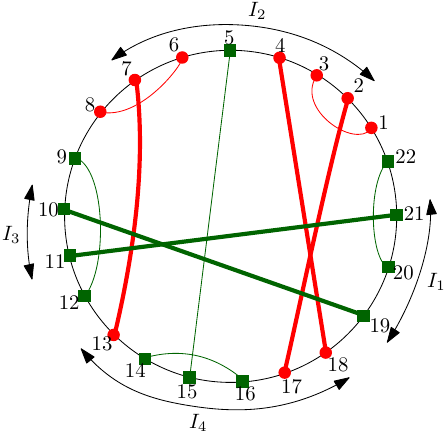} \]
\caption{Two examples of splits where the corresponding sets $S_{even}$
and $S_{odd}$ are not unions of at most two circular intervals.
By convention, green points/chords correspond to elements of $S_{odd}$/$V_{odd}$
and bold chords indicate elements of the cut vertex set $V_{odd}^{cut}$
(and similarly replacing green by red, and odd by even).
\label{fig:bad_splits}}
\end{figure}
If each of $S_{odd}$ and $S_{even}$ were a union of at most two circular intervals,
we would immediately have a decomposition as in \cref{def:decomposable_matching}
and conclude that $\match$ is decomposable.
However this is not always the case (as can be seen on the examples of \cref{fig:bad_splits}).
The strategy of proof is thus to 
define a particular type of split, called pure $4$-split,
that satisfies the following:
\begin{itemize}[itemsep=0mm,topsep=0mm,parsep=0pt,partopsep=0pt]
\item 
first, there exists a pure $4$-split as soon as $G$ is split-decomposable;
\item second, for a pure $4$-split, $S_{odd}$ and $S_{even}$
are unions of at most two circular intervals. 
\end{itemize}
The decomposability of $\match$ will follow immediately.

Since $G$ is connected, $V_{odd}^{cut}$ and $V_{even}^{cut}$ are nonempty.
By definition of cut-set, any chord in $V_{odd}^{cut}$ crosses any chord in $V_{even}^{cut}$, hence has the same amount of elements of $S_{even}^{cut}$ on each side.
As a consequence, one can show that there exist a positive integer $d$ and $4d$ nonempty sets $S^{cut}_1$, \dots, $S^{cut}_{4d}$
that appear in this order counter-clockwise around the circle and
such that:
\begin{itemize}[itemsep=0mm,topsep=0mm,parsep=0pt,partopsep=0pt]
\item $S_{odd}^{cut} = \bigcup_{j \text{ odd}} S^{cut}_j$ and 
 $S_{even}^{cut} = \bigcup_{j \text{ even}} S^{cut}_j$;
\item and any chord in the cut vertex sets $V_{even}^{cut}$ (resp.~$V_{odd}^{cut}$) goes from $S^{cut}_j$ to $S^{cut}_{j+2d}$ for some even $j \le 2d$ (resp.~odd $j \le 2d$).
\end{itemize}
We then say that the split $\{V_{odd},V_{even}\}$ is a $4d$-split.
For $j \in [4d]$, we also let $I_j$ be the smallest circular interval containing $S^{cut}_j$
(not containing  $S^{cut}_{j'}$ for $j' \ne j$).
The definition of the intervals $I_j$ is illustrated on the two examples of \cref{fig:bad_splits};
the example on the left is an $8$-split, while that on the right is a $4$-split (but not a pure $4$-split).\smallskip

{\em Subcase 3a: $d>1$.} When $d>1$, we claim that $I_j$ contains only points
of $S_{odd}$ (resp.~$S_{even}$) whenever $j$ is odd (resp.~even).
 Let us prove the claim by contradiction and assume, w.l.o.g.,
 that $I_1$ contains a point in $S_{even}$.
 Let $c$ be the chord containing this point; by construction $c$ is in
  $V_{even} \setminus V_{even}^{cut}$.
  Hence $c$ cannot cross chords of $V_{odd}^{cut}$ forcing both extremities of $c$ to be in $I_1$ (see \cref{fig:no_problemo}, left).
  The set of chords in  $V_{even}$ with extremities in $I_1$ then form a connected
  component (or several) of the graph $G$, contradicting the connectedness of $G$.
  This proves the claim that $I_j$ contains only points
of $S_{odd}$ (resp.~$S_{even}$) whenever $j$ is odd (resp.~even).
  \begin{figure}
  \[ \includegraphics{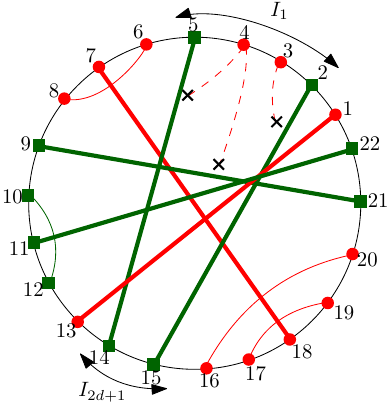} \qquad \includegraphics{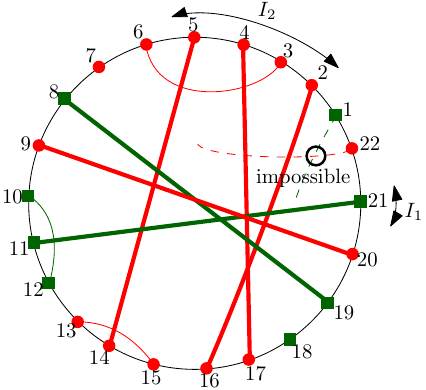} \]
  \caption{Illustration of the proof of \cref{prop:prime=indec}, Subcase 3a.
  Left: the interval $I_1$ contains only green points. Right: there is no alternation
  of colors between $I_1$ and $I_2$. \label{fig:no_problemo}
  }
  \end{figure}

   Furthermore, let us call $\vec A$ the oriented circular arc going from $I_1$ to $I_2$.
   By construction, points in $\vec A$ are either in  $S_{odd} \setminus S_{odd}^{cut}$
   or in $S_{even} \setminus S_{even}^{cut}$.
 We claim that in $\vec A$,
 we first see points of $S_{odd}$ and then points of $S_{even}$.
 Assume it is not the case, and that there exists a point $x$ of $S_{even}$
 preceding a point $y$ of $S_{odd}$ in $\vec A$.
 Since $G$ is connected, $x$ must be connected by a series of chords  to $I_2$, these chords belonging by construction to $V_{even}
 \setminus V_{even}^{cut}$,
 while $y$ must be connected by a series of chords to $I_1$, these chords belonging by construction to $V_{odd}\setminus V_{odd}^{cut}$.
 This forces a chord of $V_{even} \setminus V_{even}^{cut}$
  to cross a chord of $V_{odd}\setminus V_{odd}^{cut}$,
 which is impossible -- see \cref{fig:no_problemo}, right.
  In general, on the arc going from $I_j$ to $I_{j+1}$, 
  we first see points of $S_{even}$
  and then points of $S_{odd}$ if $j$ is even, and
  conversely if $j$ is odd.

  This implies that the circle can be cut in $4d$ circular intervals $S_1$,\dots, $S_{4d}$
  such that $S_{odd} = \bigcup_{j \text{ odd}} S_j$ and 
and $S_{even} = \bigcup_{j \text{ even}} S_j$.
Moreover, edges of the cut set go from $S_j$ to $S_{j+2d}$ for some $j$,
while edges not in the cut set go from some $S_j$ to itself.
We then set 
\[ \begin{array}{rclcrcl}
C_1 &\!\!\!=\!\!\!& S_1 \cup \dots \cup S_d, &\quad &C_2 &\!\!\!=\!\!\!& S_{d+1} \cup \dots \cup S_{2d},\\
 C_3 &\!\!\!=\!\!\!& S_{2d+1} \cup \dots \cup S_{3d},&\quad& C_4 &\!\!\!=\!\!\!& S_{3d+1} \cup \dots \cup S_{4d}.
 \end{array}\]
Up to renaming cyclically $(C_1,C_2,C_3,C_4)$ so that $1 \in C_1$,
the $C_i$'s form a partition of the circle as in \cref{def:decomposable_matching},
proving that $\match$ is decomposable.
  \smallskip
  
  {\em Subcase 3b: $d=1$.}
  When $d=1$, unlike in the previous case, 
  it might happen that there is a chord
   $c$ in $V_{odd} \setminus V_{odd}^{cut}$ 
   having one endpoint in $I_2$ and one in $I_4$
   (see \cref{fig:bad_and_good_splits}, left),
   or symmetrically a chord
   $c'$ in $V_{even} \setminus V_{even}^{cut}$ 
  going from $I_1$ and $I_3$.
  We say that the split $\{V_{odd},V_{even}\}$ is even-pure (resp.~odd-pure)
if there is no chord $c$ in $V_{odd} \setminus V_{odd}^{cut}$ 
having one endpoint in $I_2$ and one in $I_4$
(resp.~no chord $c'$ in $V_{even} \setminus V_{even}^{cut}$ 
having one endpoint in $I_1$ and one in $I_3$).
A split is {\em pure} if it is simultaneously odd-pure and even-pure.

  When the split $\{V_{odd},V_{even}\}$ is pure,
  the same argument as in Subcase 3a shows that $S_{odd}$ and $S_{even}$
  decompose as $S_1 \cup S_3$ and $S_2 \cup S_4$ respectively,
  where $S_1$, $S_2$, $S_3$, $S_4$ are circular intervals appearing in this order 
  counter-clockwise along the circle.
  Thus $\match$ is decomposable.
  
  \begin{figure}
\[\begin{array}{c}
\includegraphics{ImpureSplit}
\end{array} \quad
\begin{array}{c}
\includegraphics{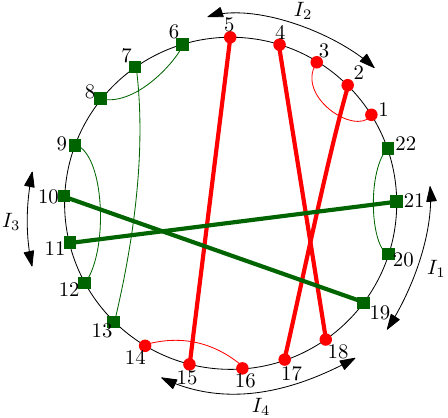}
\end{array} \]
	\captionsetup{width=\linewidth}
\caption{Left:
A split that is not even-pure, as there is a green edge from $I_2$ to $I_4$.
Modifying the split as indicated in the proof (here $c=\{5,15\}, C=\{c\} \cup \{\{14,16)\}\}$, and $V^{even}_3$ consists of the chords $\{6,8\}$ and $\{7,13\}$) yields the bicoloration of the same matching  shown on the right picture. The corresponding split is both even-pure and odd-pure.}
\label{fig:bad_and_good_splits}
\end{figure}
  We will show that one can always transform (possibly in several steps)
  an impure $4$-split into a pure one.
The notation and this part of the proof are illustrated in \cref{fig:bad_and_good_splits}.
First observe that any $4$-split is either even-pure or odd-pure
(otherwise there would be crossing chords $c$ and $c'$ in $V_{odd} \setminus V_{odd}^{cut}$ and $V_{even} \setminus V_{even}^{cut}$ respectively, which is impossible).
Without loss of generality,
we can assume that our initial split $\{V_{odd},V_{even}\}$ is odd-pure
but not even-pure.
Let $c$ be a chord in $V_{odd} \setminus V_{odd}^{cut}$ 
having one endpoint in $I_2$ and one in $I_4$.
Since $c$ is not in $V_{odd}^{cut}$, it does not cross chords in $V_{even}$. 
Thus this partitions $V_{even}$ into two mutually noncrossing nonempty sets $V_{even}^1$ and 
$V_{even}^3$, the chords of $V_{even}^1$ (resp.~$V_{even}^3$) 
being on the same side of $c$ as $I_1$ (resp.~$I_3$).
Let $C$ be the connected component of $c$ in the induced subgraph of $G$
on vertex set $V_{odd} \setminus V_{odd}^{cut}$.
Then it is easy to check that $V'_{even}=V_{even}^1 \cup C$ and 
$V'_{odd}=(V_{odd} \setminus C)  \cup  V_{even}^3$
forms a nontrivial split of $G$, which is still odd-pure. 
Moreover, the intervals $I'_2$ and $I'_4$ associated with this new split
are strictly smaller than $I_2$ and $I_4$. 
Therefore, we can iterate this transformation until finding a pure nontrivial split.
Since we have already shown that $\match$ is decomposable
whenever $G$ contains a pure nontrivial split, this ends the proof of the proposition.
\end{proof}

\subsection{Proof of  \cref{lem:XYZ-poisson}:  distribution of atypical chords}\label{Sec:AppendicePoisson}
Before recalling the statement of the lemma, we recall some notation.
For a matching $\match$, $x(\match) = \sum_{i=1}^{2n} \One[\match(i)\equiv i+1]$ is the number of chords between adjacent points,
where $\equiv$ stands for equality mod $2n$. Similarly $y(\match)=\sum_{j=1}^{2n} \One[\match(j)\equiv j+2]$. Finally 
$$
z(\match)=\sum_{\substack{1\leq k<\ell\leq 2n \\ \ell-k\not\equiv \pm 1}} \One\big[\set{\match(k),\match(k+1)}\equiv\set{\ell,\ell+1}\big];
$$
\emph{i.e.}~$z(\match)$ counts pairs of consecutive points matched to another pair of consecutive points.
The definitions of $x(\match)$, $y(\match)$ and $z(\match)$ are illustrated on \cref{Fig:XnYnZn}.
Finally, recall that $M_n$ denotes a uniform random matching of size $n$, and $X_n=x(M_n)$, $Y_n=y(M_n)$ and $Z_n=z(M_n)$. 
\medskip

\noindent{\bf Lemma \ref{lem:XYZ-poisson}.}
\emph{
The triple $(X_n,Y_n,Z_n)$ converges in distribution towards a triple of independent Poisson random variables with mean 1. 
}

\begin{proof}
It is enough to prove that the joint factorial moments of $X_n$, $Y_n$ and $Z_n$ 
tend to $1$
(see pp.~60-62 in \cite{Remco} for a review on Poisson convergence
 and joint factorial moments; in particular \cite[Theorem 2.6]{Remco} states
 that the convergence of joint factorial moments implies the joint convergence in distribution).
We write $(x)_r=x(x-1)\cdots (x-r+1)$ for factorial powers.
For integers $r,s,t \ge 1$, the joint factorial moments expand as
\begin{align}\label{eq:joint_factorial_moments}
 \mathbb E\big[ (X_n)_r (Y_n)_s (Z_n)_t  \big]
   &= \sum_{i_1 \dots i_r,j_1 \dots j_s
   \atop k_1<\ell_1,\dots,k_t<\ell_t}  P_{\mathbf{i},\mathbf{j},\mathbf{k}}, \\
\text{where }P_{\mathbf{i},\mathbf{j},\mathbf{k}}&= \mathbb P\left(
   \begin{cases}
   M_n(i_\alpha)\equiv i_\alpha+1 &\forall \alpha \le r \\
     M_n(j_\beta)\equiv j_\beta+2 &\forall \beta \le s \\
     \set{M_n(k_\gamma),M_n(k_\gamma+1)}\equiv\set{\ell_\gamma,\ell_\gamma+1}&\forall \gamma \le t 
   \end{cases}\right) \notag
\end{align}
   and the sum is taken over lists $\mathbf{i}=(i_1,\dots,i_r)$, $\mathbf{j}=(j_1,\dots,j_s)$ and $\mathbf{k}=((k_1,\ell_1),\dots,
   (k_t,\ell_t))$ such that
   \begin{itemize}
   \item all $i_\alpha$'s are  distinct;
   \item all $j_\beta$'s are  distinct;
   \item all pairs $(k_\gamma,\ell_\gamma)$ are  distinct and furthermore  $\ell_\gamma-k_\gamma \not\equiv \pm 1$, for every $\gamma \le t$.
   \end{itemize}
        
   In the above sum, let us consider first the summands for which all indices   $i_\alpha$, $ i_\alpha+1$, $j_\beta$,
   $j_\beta+2$, $k_\gamma$, $k_\gamma+1$, $\ell_\gamma$ and $\ell_\gamma+1$ are distinct. 
   We call such terms $P_{\mathbf{i},\mathbf{j},\mathbf{k}}$ {\em nice}, while other terms are referred to as {\em painful}.
   For each nice term, 
   \[
    P_{\mathbf{i},\mathbf{j},\mathbf{k}}=2^t\frac{m_{n-r-s-2t}}{m_n}.
    \]
   Indeed, for each $\gamma$ we may choose whether $M_n(k_\gamma)=\ell_\gamma+1$ and 
   $M_n(k_\gamma+1)=\ell_\gamma$ or the converse, explaining the factor $2^t$.
   Additionally, once these choices are made, 
   the chords involving indices of ${\mathbf{i},\mathbf{j},\mathbf{k}}$ are fixed, and
   the remaining chords induce a uniform matching of size $n-r-s-2t$.
   Using that $m_{n-1}/m_n = 1/(2n-1) \sim 1/(2n)$, we have that for fixed $r,s,t$ and
   for each nice term $P_{\mathbf{i},\mathbf{j},\mathbf{k}}$ the following holds:
   \begin{equation}\label{eq:asymp_Pijk}
    P_{\mathbf{i},\mathbf{j},\mathbf{k}}\stackrel{n\to +\infty}{\sim} 2^t\, (2n)^{-r-s-2t}.
    \end{equation}

 We now want to estimate the number $\mathcal{N}_n(r,s,t)$ of nice terms.
 Let us remark that if we take $i_1,\dots,i_r$ and $j_1,\dots,j_s$ uniformly in $[2n]$,
 and $(k_1,\ell_1)$, \dots, $(k_t,\ell_t)$ uniformly in $[2n]^2$ conditioned to satisfying $k_\gamma <\ell_\gamma$, all independent to 
 each other, then $\mathbf{i},\mathbf{j},\mathbf{k}$ is the index of a nice term
 with probability tending to $1$ as $n$ tends to $+\infty$. Indeed, a fixed number (here $r+s+2t$) of uniform integers
 in $[2n]$ contains neither repetitions, nor adjacent points with probability tending to $1$.
 Hence, as $n$ tends to $+\infty$, we have
 \[
\mathcal{N}_n(r,s,t) \sim (2n)^r (2n)^s \binom{2n}{2}^{t} \sim 2^{-t} \, (2n)^{r+s+2t}.
 \]
  We conclude that the total contribution of nice terms to the sum in
   \cref{eq:joint_factorial_moments} tends to $1$.\medskip

We now prove that the total contribution of painful terms is asymptotically negligible.
With a triple of lists $(\mathbf{i},\mathbf{j},\mathbf{k})$ as above, we associate a graph
$G_{\mathbf{i},\mathbf{j},\mathbf{k}}$ encoding {\em coincidences} as follows.
\begin{itemize}
\item Its vertex set is 
\[ \{a_\alpha, \alpha \le r\} \cup \{b_\beta, \beta \le s\} \cup \{c_\gamma, \gamma \le t\}
\cup \{d_\gamma, \gamma \le t\}.\]
Each $a_\alpha$ (resp.~$b_\beta$, $c_\gamma$, $d_\gamma$) is a
formal symbol representing the set $\{i_\alpha,i_\alpha+1\}$ (resp.~the set $\{j_\beta,j_\beta+2\}$,
$\{k_\gamma,k_\gamma+1\}$, $\{\ell_\gamma,\ell_\gamma+1\}$).
\item There is an edge between two vertices when the corresponding sets have 
a nonempty intersection.
\end{itemize}
Nice terms are those for which $G_{\mathbf{i},\mathbf{j},\mathbf{k}}$ is the empty graph.
For a nonempty graph $G$, let us denote by $\mathcal N_n^G$ the number of triples
$(\mathbf{i},\mathbf{j},\mathbf{k})$ with $G_{\mathbf{i},\mathbf{j},\mathbf{k}} = G$.
We have $$\mathcal N_n^G=\mathcal O\big(n^{\cc(G)}\big),$$
 where $\cc(G)$ is the number of connected
components of $G$. Indeed, one can choose freely the value
of $i_\alpha$ (or $j_\beta$, $k_\gamma$, $\ell_\gamma$) for one vertex in each 
connected component of $G$. Then there are only finitely many choices for the value 
of $i_\alpha$ (or $j_\beta$, $k_\gamma$, $\ell_\gamma$) for other vertices in the same component.

  We now discuss the value of $P_{\mathbf{i},\mathbf{j},\mathbf{k}}$.
  In some cases, \emph{e.g.}~if $\{i_\alpha,i_\alpha+1\} \cap \{j_\beta,j_\beta+2\}\ne \emptyset$ 
  for some $\alpha,\beta$, the conditions in the definition of $P_{\mathbf{i},\mathbf{j},\mathbf{k}}$ are incompatible and
  $P_{\mathbf{i},\mathbf{j},\mathbf{k}}=0$.
  Otherwise the conditions define a configuration of chords that the random matching $M_n$ should contain (or more precisely
  $M_n$ should contain a configuration
  among a finite number of possible ones, as for each $\gamma$, one can choose
  whether $k_\gamma$ is connected to $\ell_\gamma$ and 
  $k_\gamma+1$ to $\ell_\gamma+1$ or conversely).
  With a similar reasoning as in \cref{eq:asymp_Pijk}, we have 
  $$P_{\mathbf{i},\mathbf{j},\mathbf{k}}=\mathcal O\big(n^{-\chords(\mathbf{i},\mathbf{j},\mathbf{k})}\big),$$
  where $\chords(\mathbf{i},\mathbf{j},\mathbf{k})$ is the number of chords in this configuration.
  
  We claim that for any triple $(\mathbf{i},\mathbf{j},\mathbf{k})$
  such that $G_{\mathbf{i},\mathbf{j},\mathbf{k}}$ is nonempty
  and $P_{\mathbf{i},\mathbf{j},\mathbf{k}} \neq 0$, we have
  \begin{equation}\label{eq:chords_cc}
  \chords(\mathbf{i},\mathbf{j},\mathbf{k}) > \cc(G_{\mathbf{i},\mathbf{j},\mathbf{k}}).
  \end{equation}
Assuming temporarily the claim, for any nonempty graph $G$,
 the total contribution of triples $(\mathbf{i},\mathbf{j},\mathbf{k})$ with $G_{\mathbf{i},\mathbf{j},\mathbf{k}}=G$ to \cref{eq:joint_factorial_moments} is negligible.
 Hence the total contribution of painful terms is negligible
 and $\mathbb E\big[ (X_n)_r (Y_n)_s (Z_n)_t  \big]$ tends to $1$ as desired.
 
 It only remains to show \cref{eq:chords_cc}. Configurations $(\mathbf{i},\mathbf{j},\mathbf{k})$ with a nonempty graph $G_{\mathbf{i},\mathbf{j},\mathbf{k}}$ can be decomposed
 into basic types of coincidences (and chords corresponding to isolated vertices in $G_{\mathbf{i},\mathbf{j},\mathbf{k}}$), which are represented on \cref{Fig:Coincidence}
 and on which \eqref{eq:chords_cc} is easy to check.
 This concludes the proof of the lemma.
   \begin{figure}[t]
\begin{center}
\includegraphics[width=\textwidth]{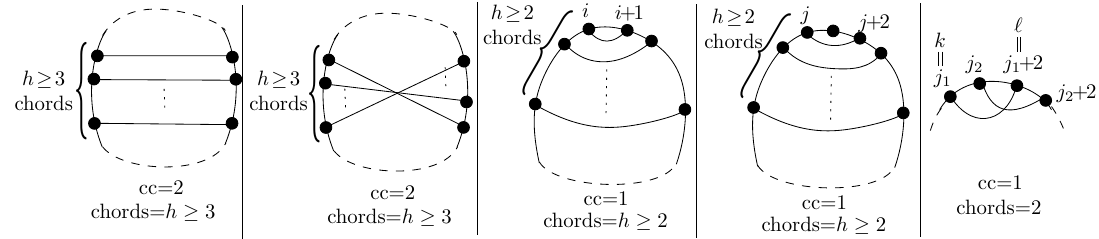}
\end{center}
\caption{Basic configurations with coincidences (every pair of adjacent parallel or crossing lines correspond to some $k_\gamma$ and $\ell_\gamma$). 
All satisfy $\cc<\chords$.
}
\label{Fig:Coincidence}
\end{figure}   
\end{proof}

} 

\end{document}